\documentclass[a4paper]{article}
\pdfoutput=1
\usepackage{amsthm,amssymb,amsmath,enumerate,graphicx}
\advance\textheight by -1cm

\newtheorem{definition}{Definition}[section]
\newtheorem{proposition}[definition]{Proposition}
\newtheorem{theorem}[definition]{Theorem}
\newtheorem{corollary}[definition]{Corollary}
\newtheorem{lemma}[definition]{Lemma}

\newtheorem{example}[definition]{Example}

\newcommand{\COMMENT}[1]{}
\newcommand{\<}[1]{\vadjust{\vbox to 0pt{\vss\vskip-8pt\leftline{%
     \llap{\rm\hbox{\vbox{\pretolerance=-1
     \doublehyphendemerits=0\finalhyphendemerits=0
     \hsize16truemm\tolerance=10000\small
     \lineskip=0pt\lineskiplimit=0pt
     \rightskip=0pt plus16truemm\baselineskip8pt\noindent
     \hskip0pt        
     #1\endgraf}\hskip7truemm}}}\vss}}}
\renewcommand{\<}[1]{}

\newcommand{\sm}{\smallsetminus}
\newcommand{\sub}{\subseteq}
\newcommand{\subne}{\subsetneq}
\newcommand{\supe}{\supseteq}
\newcommand{\es}{\emptyset}
\newcommand{\B}{\ensuremath\mathcal{B}}
\newcommand{\C}{\ensuremath\mathcal{C}}
\newcommand{\E}{\ensuremath\mathcal{E}}
\newcommand{\Emax}{\ensuremath\mathcal{E^{\rm max}}}
\newcommand{\F}{\mathbb F}
\newcommand{\I}{\ensuremath\mathcal{I}}
\newcommand{\Imax}{\ensuremath\mathcal{I^{\rm max}}}
\newcommand{\N}{\mathbb N}
\newcommand{\R}{\mathbb R}

\newcommand{\Xbar}{{\overline X}}
\newcommand{\Bbar}{{\overline B}}

\newcommand{\dcl}[1]{\ensuremath\lceil#1\rceil}
\newcommand{\cl}{\ensuremath{{\rm cl}}}

\newcommand{\mfc}{\ensuremath{M_{\rm FC}}}
\newcommand{\mc}{\ensuremath{M_{\rm TC}}}
\newcommand{\Ccal}{{\cal C}}
\newcommand{\Dcal}{{\cal D}}

\newcommand{\mfb}{\ensuremath{M_{\rm FB}}}
\newcommand{\mb}{\ensuremath{M_{\rm B}}}

\title{Axioms for infinite matroids}
\author{Henning Bruhn \and Reinhard Diestel \and Matthias Kriesell
\and Rudi Pendavingh \and Paul Wollan}
\date{}

\begin{document}
\maketitle

\begin{abstract}\noindent
We give axiomatic foundations for infinite matroids with duality, in terms of independent sets, bases, circuits, closure and rank. Continuing work of Higgs and Oxley, this completes the solution to a problem of Rado of 1966.\looseness=-1
\end{abstract}

\section*{Introduction}

Traditionally, infinite matroids are most often defined like finite ones,%
   \footnote{The augmentation axiom is required only for finite sets: given independent sets $I,I'$ with $|I| < |I'| < \infty$, there is an $x\in I'\sm I$ such that $I+x$ is again independent.} with the following additional axiom:

\begin{itemize}
\item[\rm (I4)] An infinite set is independent as soon as all its finite subsets are independent.
\end{itemize}

\noindent
We shall call such set systems {\em finitary matroids\/}.

The additional axiom (I4) reflects the notion of linear independence in vector spaces, and also the absence of (finite) circuits from a set of edges in a graph. Conversely, it is a direct consequence of (I4) that circuits, defined as minimal dependent sets, are finite.

Historically, the introduction of axiom (I4) coincided with the discovery in the first half of the 20th century of what is now called the `compactness technique' in infinite combinatorics. And indeed, if we define infinite matroids via~(I4), we find that while not all statements about finite matroids extend to infinite ones, those that do tend to extend `by compactness'.%
   \footnote{Among the infinite matroids introduced in this paper, the finitary ones will be precisely those whose set of independent sets is closed in the (compact) power set of their ground set.}
   Today, compactness proofs in infinite combinatorics are considered standard. The fact that (I4) restricts infinite matroids to those structures whose essential properties can be derived by (mere) compactness is no longer seen as an asset, but as a sign of the limited added diversity of infinite matroid theory over the finite theory.

The axiom (I4) is also the crudest possible way of specifying the infinite independent sets in an infinite matroid, given its finite independent sets. Indeed, suppose we have a collection of finite subsets of an infinite set that satisfy the independence axioms for finite matroids. Which infinite sets can we declare as independent and remain consistent with those axioms? Since independence should be hereditary under taking subsets, we can only take sets whose finite subsets are independent. Axiom~(I4) simply tells us to take all of these.%
   \COMMENT{}

The most devastating consequence of (I4), however, is that it spoils duality, one of the key features of finite matroid theory. For example, the cocircuits of an infinite uniform matroid of rank~$k$ would be the sets missing exactly $k-1$ points; since these sets are infinite, however, they cannot be the circuits of another finitary matroid. Similarly, every bond of an infinite graph would be a circuit in any dual of its cycle matroid---a set of edges minimal with the property of containing an edge from every spanning tree---%
   \COMMENT{}%
   but these sets can be infinite and hence will not be the circuits of a finitary matroid.

This situation prompted Rado in 1966 to ask for the development of a theory of non-finitary infinite matroids with duality~\cite[Problem P531]{Rado66matroids}.%
   \COMMENT{}
   The collection of independent sets in such matroids would have to be a more subtly chosen subcollection of the sets specified by~(I4), balanced carefully to make duality possible. The collection of circuits for such matroids would necessarily have to allow for infinite circuits.

Rado's challenge caused some serious activity in the late 1960s (see e.g.\ \cite{Oxley78} for references), in which several authors suggested numerous possible notions of infinite matroids.%
   \COMMENT{}
   Each of these highlighted one of the aspects of finite matroids (usually closure), some had duality built in by force, but none came with a set of axioms similar to those known from finite matroids: axioms that would make these structures the models of what was visibly a theory of finite and infinite matroids. This situation led to the popular belief, common to this day, that Rado's problem may have no solution: that there may be no theory of infinite matroids with all the usual aspects including duality.%
   \footnote{Compare Oxley~\cite[p.$\,$68]{OxleyBook}.}%
   \COMMENT{}

Despite these negative expectations, those early activities made an important contribution: they identified necessary conditions which any theory of infinite matroids would have to satisfy. Eventually, Oxley~\cite{Oxley92} proved that any theory of infinite matroids with duality and minors as we know them would have as its models certain%
   \COMMENT{}
   structures that Higgs~\cite{Higgs69duality} had proposed as `B-matroids'. Although Oxley~\cite{Oxley78,OxleyThesis} had earlier found a set of axioms for these `B-matroids' resembling a mixture of independence and base axioms, it remained an open problem whether axiom sets of the kind known from finite matroids existed to capture these structures---axiom sets that would make them accessible to the tools and techniques that finite matroid theory had developed over the years.

Our aim in this paper is to finally settle Rado's problem, in the affirmative. We propose five equivalent sets of matroid axioms, in terms of independent sets, bases, circuits, closure and rank, that make duality possible. They will allow for infinite circuits, but default to finitary matroids when their circuits happen to be finite. Duality will work as familiar from finite matroids: the cobases are the complements of bases, and there are well-defined and dual operations of contraction and deletion extending the familiar finite operations.

Generic examples of these matroids abound: they include the duals of all finitary matroids, a~vast class of structures that can now also be described in matroidal terms. (These duals are not normally finitary.) For example, there are now matroid duals of vector spaces. This may, for the first time, facilitate a genuine use of matroidal techniques in linear algebra: an idea that motivated the creation of matroid theory, but which never came to fruition since matroid proofs without duality can usually be carried out in the vector space directly. The matroid duals of vector spaces and their submatroids have been characterized recently by Afzali and Bowler~\cite{AfzaliBowler12}, based on the axioms introduced in this paper.\looseness=-1

One particularly striking example where the now existing dual of a finitary matroid was found to describe, and shed a new light on, a previously known important structure outside matroid theory is the following. When a graph~$G$ is infinite (but locally finite), its first homology is best described~-- in the sense that it captures similar structural features of $G$ as does its simplicial first homology when $G$ is finite~-- not by the (simplicial or singular) first homology of $G$ itself, but by a non-singular homology theory based on possibly infinite sums of the edge sets of topological circles in its Freudenthal compactification~$|G|$; see~\cite{Hom1, FundGp, Hom2}. It turns out that these edge sets of circles in $|G|$ are precisely the circuits of the dual of the finite-bond matroid of~$G$, a finitary matroid that has long been known but which under (I4) had no dual. Considering that there is nothing topological in either the definition of the finite-bond matroid or in our infinite matroid axioms, the fact that these (edge sets of) topological circles come up as the cocircuits of finitary matroids seems to point to some deeper connections of the two fields than meets the eye.

There are also some `primal' infinite matroids that occur naturally outside matroid theory. For example, infinite matroids can now describe the duality of infinite graphs, which is more subtle than finite graph duality:%
   \COMMENT{}
   between their bonds and finite circuits, between their finite bonds and their topological circuits (as earlier), and in many new ways located between these two.%
   \COMMENT{}
   Planar infinite graphs can now be characterized by the duals of their cycle matroids, just as the finite planar graphs are characterized by their matroid duals via Whitney's theorem. There are also some algebraic examples, such as from simplicial homology. Further primal examples can be found in the existing literature on Higgs's `B-matroids', see e.g.~\cite{Bean76, Higgs69duality, Oxley78, LasVergnas71,Wojciechowski05}.\looseness=-1

Since our paper was first made available in preprint form~\cite{InfMatroidAxioms}, various authors have built on it to extend some of the classical results of finite matroid theory to infinite matroids, or to relate such infinite extensions to well-known conjectures outside matroid theory. Such extensions include Whitney's characterization of planar graphs~\cite{Whitney32, BruhnDiestelMatroidsGraphs}, excluded minor characterizations of classes of representable and graphic matroids~\cite{BC12:excluded_matroid_minors, BCC:graphic_matroids}, Tutte's linking theorem~\cite{TutteMenger65, BruhnWollanConInfMatroids, B13:TutteConnectivity}, the decomposition theorem of Cunningham-Edmonds and Seymour~\cite{CE, Seymour, DiestelHorevPostleDecomposition, BC13:Ubiquity}, the matroid union theorem of Rado, Nash-Williams and Edmonds~\cite{Rado42, NashWilliamsMatroidUnion, EdmondsMatroidUnion, A-HCF11:matroidunion}, cases of Edmonds's matroid intersection theorem~\cite{Edmonds1970, Edmonds1979, A-HCF11:matroidintersection, BowlerCarmesinMatroidIntersection, BC13:inter_psi}, and the base packing and covering theorems of Edmonds~\cite{Edmonds1979, A-HCF11:matroidunion, BowlerCarmesinMatroidIntersection} (in part anticipated by Horn, Nash-Williams, and Tutte). 

When developing our axioms we faced two challenges: to avoid the use of cardinalities, and to deal with limits. As concerns the latter, we want every independent set to extend to a base (so that there can be an equivalent set of base axioms, in which independent sets are defined as subsets of bases), and we want every dependent set to contain a circuit (so that there can be an equivalent set of circuit axioms, in which independent sets are defined as the sets not containing a circuit). It turns out that we have to require one of these as an additional axiom, but the other will then follow.

Devising axioms without reference to cardinalities is a more serious challenge. Consider two independent sets $I_1,I_2$ in a finite matroid. How can we translate the assumption, made in the third of the standard independence axioms, that $|I_1|<|I_2|$? If $I_1\sub I_2$, this is equivalent (for finite sets) to $I_1\subne I_2$, and we can use the latter statement instead. But if $I_1\not\sub I_2$, the only way to designate $I_1$ as `smaller' and $I_2$ as `larger' is to assume that $I_2$ is maximal among all the independent sets while $I_1$ is not---a much stronger statement (for finite matroids) that fails to capture size differences among non-maximal independent sets. Nevertheless, we shall see that this distinction will be enough.

Similarly, the cardinality of bases is too crude a measure for rank when it is infinite: although bases are equicardinal also for our matroids~\cite{Higgs69equicard}, deleting one element of an infinite independent set should reduce its rank by~1 but does not reduce its cardinality. Our solution to this problem will be to measure not absolute but `relative' rank: for finite matroids, this would be the amount by which the rank of a set $A$ exceeds that of a given subset~$B\sub A$, and it will be 1 in the above example. It turns out that the usual rank axioms can be rephrased for finite matroids in terms of such relative rank, in a way that yields an axiom system that becomes equivalent to the other systems also for infinite matroids. 

Our paper is organized as follows. In Section~\ref{axioms} we state our axiom systems for infinite matroids, and provisionally define infinite matroids as set systems satisfying the independence axioms. Section~\ref{examples} is devoted to examples of infinite matroids that are not necessarily finitary. We have also included some references to further results and examples of infinite matroids (by other authors) that have become possible since we first presented our axioms in~\cite{InfMatroidAxioms}. In Section~\ref{basics} we establish a minimum of basic properties of our infinite matroids (including duality and the existence of minors): those that will enable us in Section~\ref{eq} to prove that the independence axioms are in fact equivalent to the other axiomatic systems proposed in Section~\ref{axioms}, as well as to the traditional axioms when the matroid is finitary. Section~\ref{alt} provides some alternative axiom systems, which are more technical to state but may be easier to verify, and are hence worth knowing. We also include the mixed set of independence and base axioms developed for Higgs's `B-matroids' by Oxley~\cite{Oxley78,OxleyThesis}. In Section~\ref{non-matroids}, finally, we illustrate our axioms by examples of set systems that narrowly fail to satisfy them, by missing just one axiom each. In particular, our axioms are shown to be independent. 

Any matroid terminology not explained below is taken from Oxley~\cite{OxleyBook}. Terms used in our infinite graph examples, such as ends of graphs and their topology, are defined in~\cite{DiestelBook10noEE}. Let $E$ be any set, finite or infinite. This set~$E$ will be the default ground set for all matroids considered in this paper. We write $\overline X := E\sm X$ for complements of sets $X\sub E$, and $2^E$ for the power set of~$E$. The set of all pairs $(A,B)$ such that $B\subseteq A\subseteq E$ will be denoted by $(2^E\times 2^E)_\subseteq$; for its elements we usually write $(A|B)$ instead of~$(A,B)$.%
   \COMMENT{}
   Unless otherwise mentioned, the terms `minimal' and `maximal' refer to set inclusion. Given $\E\sub 2^E$, we write $\Emax$ for the set of maximal elements of~$\E$, and $\dcl{\E}$ for the {\em down-closure} of~$\E$, the set of subsets of elements of~$\E$. For $F\sub E$ and $x\in E$, we abbreviate $F\sm\{x\}$ to $F-x$ and $F\cup\{x\}$ to~$F+x$. We shall not distinguish between infinite cardinalities and denote all these by~$\infty$; in particular, we shall write $|A|=|B|$ for any two infinite sets $A$ and~$B$. The set $\N$ contains~0.

\section{Axiom systems for infinite matroids}\label{axioms}

In this section we present our five systems of axioms for infinite matroids. They are stated, respectively, in terms of independent sets, bases, closure, circuits and rank.

One central axiom that features in all these systems is that every independent set extends to a maximal one, inside any restriction $X\sub E$.%
   \footnote{Interestingly, we shall not need to require that every dependent set contains a minimal one. We need that too, but will be able to prove it; see Section~\ref{basics}.}
    The notion of what constitutes an independent set, however, will depend on the type of axioms under consideration. We therefore state this extension axiom in more general form right away, without reference to independence, so as to be able to refer to it later from within different contexts.

Let $\I\sub 2^E$. The following statement describes a possible property of~$\I$.

\begin{itemize}
\item[\rm (M)] Whenever $I\sub X\sub E$ and $I\in\I$, the set $\{\,I'\in\I\mid I\sub I'\sub X\,\}$ has a maximal element.
\end{itemize}
Note that the maximal superset of $I$ in~$\I\cap 2^X$ whose existence is asserted in~(M) need not lie in~$\Imax$.

\subsection{Independence axioms}\label{independenceaxioms}

The following statements about a set $\I\sub 2^E$ are our {\em independence axioms\/}:
\begin{itemize}
\item[\rm (I1)] $\emptyset\in\mathcal I$.
\item[\rm (I2)] $\dcl{\I}=\I$, i.e., $\I$ is closed under taking subsets.
\item[\rm (I3)] For all $I\in\I\sm\Imax$ and $I'\in\Imax$ there is an $x\in I'\sm I$ such that $I+x\in\I$.\looseness=-1
\item[\rm (IM)] $\I$ satisfies~(M).
\end{itemize}

\noindent
We remark that although (IM) formally depends on our choice of $E$ as well as that of~$\I$, this dependence on $E$ is not crucial: if $\I$ satisfies (IM) for some set~$E$ large enough that $E\supe\bigcup\I$, it does so for every such set~$E'$.%
   \COMMENT{}

\medbreak

When a set $\I\sub 2^E$ satisfies the independence axioms, we call the pair $(E,\I)$ a {\em matroid\/} on~$E$. We then call every element of $\I$ an {\em independent\/} set, every element of $2^E\sm\I$ a {\em dependent\/} set, the maximal independent sets {\em bases}, and the minimal dependent sets {\em circuits}. The $2^E\to 2^E$ function mapping a set $X\sub E$ to the set
 $$\cl(X) := X\cup \{\,x\mid \exists\, I\sub X\colon I\in\I\ \text{but}\ I+x\notin\I\,\} $$
will be called the {\em closure operator\/} on $2^E$ {\em associated with~$\I$.} The $(2^E\times 2^E)_\subseteq\to \N\cup\{\infty\}$ function $r$ that maps a pair $A\supe B$ of subsets of $E$ to
 $$r(A|B) := \max\, \{\,|I\sm J| : I\supe J,\ I\in\I\cap 2^A,\ J\text{ maximal in }\I\cap 2^B\}$$
will be called the {\em relative rank function\/} on the subsets of $E$ {\em associated with~$\I$.} We shall see in Section~\ref{basics} that this maximum is always attained and independent of the choice of~$J$ (Lemma~\ref{rwelldefined}).

\subsection{Base axioms}

The following statements about a set $\B\sub 2^E$ are our {\em base axioms\/}:
\begin{itemize}
\item[\rm (B1)] $\B\ne\es$.
\item[\rm (B2)] Whenever $B_1,B_2\in\B$ and $x\in B_1\sm B_2$, there is an element $y$ of~$B_2\sm B_1$ such that $(B_1-x)+y\in\B$.\COMMENT{}
\item[\rm (BM)] The set $\I:=\dcl{\B}$ of all {\em $\B$-independent\/} sets satisfies~(M).
\end{itemize}

\subsection{Closure axioms}

The following statements about a function $\cl\colon 2^E\to 2^E$ are our {\em closure axioms\/}:
\begin{itemize}
\item[\rm (CL1)] For all $X\sub E$ we have $X\sub \cl(X)$.
\item[\rm (CL2)] For all $X\sub Y\sub E$ we have $\cl(X)\sub \cl(Y)$.
\item[\rm (CL3)] For all $X\sub E$ we have $\cl(\cl(X)) = \cl(X)$.
\item[\rm (CL4)] For all $Z\sub E$ and $x,y\in E$, if $y\in\cl(Z+x)\sm\cl(Z)$ then $x\in\cl(Z+y)$.
\item[\rm (CLM)] The set $\I$ of all \cl-{\em independent\/} sets satisfies~(M). These are the sets $I\sub E$ such that $x\notin\cl(I-x)$ for all $x\in I$.
\end{itemize}

\subsection{Circuit axioms}

The following statements about a set $\C\sub 2^E$ are our {\em circuit axioms\/}:
\begin{itemize}
\item[\rm (C1)] $\es\notin\C$.
\item[\rm (C2)] No element of $\C$ is a subset of another.
\item[\rm (C3)] Whenever $X\sub C\in\C$ and $(C_x\mid x\in X)$ is a family of elements of~$\C$ such that $x\in C_y\Leftrightarrow x=y$ for all $x,y\in X$, then for every $z \in C\sm \left( \bigcup_{x \in X} C_x\right)$ there exists an element $C'\in\C$ such that $z\in C'\sub \left(C\cup  \bigcup_{x \in X} C_x\right) \sm X$.%
   \COMMENT{}
\item[\rm (CM)] The set $\I$ of all {\em $\C$-independent\/} sets satisfies~(M). These are the sets $I\sub E$ such that $C\not\sub I$ for all $C\in\C$.
\end{itemize}
Axiom (C3) defaults for $|X|=1$ to the usual (`strong')%
   \COMMENT{}
   circuit elimination axiom for finite matroids. In particular, it implies that adding an element to a base creates at most one circuit;%
   \COMMENT{}
   the existence of such a ({\em fundamental\/}) circuit will follow from Lemma~\ref{C0}. For $|X|>1$, the inclusion of a specified element $z$ in $C'$ is not just a convenience%
   \COMMENT{}
   but essential: without it, the statement would in general be false even for finite matroids. (Take $X=C$ to be the rim of a wheel in its cycle matroid.) We shall see in Section~\ref{non-matroids} that the usual finite circuit elimination axiom is too weak to guarantee a matroid (Example~\ref{Bean}).

\subsection{Rank axioms}

The following statements about a function $r\colon (2^E\times 2^E)_\subseteq\to \N\cup\{\infty\}$ are our (relative) {\em rank axioms\/}:
\begin{itemize}
\item[\rm (R1)] For all $B\sub A\sub E$ we have $r(A|B)\le |A\sm B|$.
\item[\rm (R2)] For all $A,B\sub E$ we have $r(A|A\cap B)\ge r(A\cup B|B)$.
\item[\rm (R3)] For all $C\sub B\sub A\sub E$ we have $r(A|C) = r(A|B) + r(B|C)$.
\item[\rm (R4)] For all families $(A_\gamma)$ and $B$ such that $B\sub A_\gamma\sub E$ and $r(A_\gamma|B)=0$ for all~$\gamma$, we have $r(A|B)=0$ for $A := \bigcup_\gamma A_\gamma$.
\item[\rm (RM)] The set $\I$ of all $r$-{\em independent\/} sets satisfies~(M). These are the sets $I\sub E$ such that $r(I|I-x) > 0$ for all $x\in I$.
\end{itemize}

For finite matroids, these axioms (with (R4) and (RM) becoming redundant) are easily seen to be tantamount to the usual axioms for an absolute rank function~$R$ derived as $R(A) := r(A|\es)$, or conversely with ${r(A|B) := R(A) - R(B)}$%
   \COMMENT{}
   for $B\sub A$.%
   \COMMENT{}

\section{Examples}\label{examples}%
   \COMMENT{}

The purpose of this section is to show that the infinite matroids just defined do occur in nature: we give a small collection of natural examples from contexts in which, working on other problems, we encountered these matroids, and which made us look for a general definition.

Before we start, let us note that, for finite set systems, our definition of a matroid coincides with the usual definition. Indeed, finite matroids defined as usual are matroids in our sense: this is most easily seen in terms of our base or closure axioms, which for finite $E$ coincide with the usual base or closure axioms. More generally, all traditional finitary matroids%
   \COMMENT{}
  are matroids in our sense, as axiom (IM) follows by Zorn's Lemma; see Corollary~\ref{OldFinitaryAreMatroids} for an explicit proof.\looseness=-1

Conversely, if a matroid in our sense happens to be finite or finitary (i.e., satisfies~(I4) in addition to our axioms),%
   \COMMENT{}
   it also satisfies the usual axioms for finite or finitary matroids: the finite augmentation axiom (see the introduction) is easy to deduce from Lemma~\ref{card} below, applied in our matroid's restriction to $I\cup I'$.%
   \COMMENT{}
    Since every dependent set contains a circuit (Lemma~\ref{C0}), a matroid in our sense is finitary if and only if it has only finite circuits.

In what follows we shall concentrate on non-finitary matroids.

\subsection{Generic non-finitary matroids}

Since classical finitary matroids are matroids in our sense, and our matroids have duals, we at once have a large class of new matroids: duals of finitary matroids that are not themselves finitary. We already saw an example in the introduction: the duals of uniform matroids of finite rank. We remark that having a non-finitary dual is the rule rather than the exception for a finitary matroid: Las Vergnas~\cite{LasVergnas71} and Bean~\cite{Bean76} showed that the only finitary matroids with finitary duals are the direct sums of finite matroids.%
   \COMMENT{}

\subsection{Cycle and bond matroids in graphs}\label{SecCycleBond}

There are two standard matroids associated with a graph $G$, both finitary: the {\em finite-cycle matroid}~$\mfc(G)$ whose circuits are the edge sets of the (finite) cycles of~$G$, and the {\em finite-bond matroid}~$\mfb(G)$ whose circuits are the finite bonds of~$G$. (A~{\em bond\/} is a minimal non-empty cut.) In a finite graph these two 
matroids are dual.

When $G$ is infinite, the dual of $\mfc(G)$ is not $\mfb(G)$ but the full {\em bond matroid}~$\mb(G)$. This is the matroid whose circuits are all the bonds of~$G$, finite or infinite: these, as is easy to show, are the minimal edge sets meeting all the spanning trees of~$G$ (connected), the bases of~$\mfc(G)$.%
   \COMMENT{}
   Similarly, the dual of $\mfb(G)$ is no longer~$\mfc(G)$ but a matroid~$\mc(G)$ whose circuits can be infinite.

Surprisingly, this matroid $\mc(G)$ has a topological characterization~\cite{BruhnDiestelMatroidsGraphs}. When $G$ is connected and locally finite, it is particularly natural: its circuits are the edge sets of the topological circles in~$|G|$, the compact topological space obtained from $G$ by adding its ends.%
   \footnote{This space, also known as the Freudenthal compactification of~$G$, is the natural setting for most problems about locally finite graphs that involve paths and cycles. It has been extensively studied; see~\cite{DiestelBook10noEE, RDsBanffSurvey} for an introduction and overview.}
   Its bases are the edge sets of the {\em topological spanning trees\/} of~$G$, the arc-connected standard subspaces of $|G|$ that contain every vertex (and every end) but lose their connectedness if any edge is deleted.

\begin{theorem}$\!\!${\rm\cite{BruhnDiestelMatroidsGraphs}}\label{graphdualslocfin}
Let $G$ be a locally finite connected graph.
   \begin{enumerate}[\rm (i)]
   \item The dual of its finite-bond matroid $\mfb(G)$ is the matroid $\mc(G)$ whose circuits are the edge sets of the topological circles in $|G|$ and whose bases are the edge sets of the topological spanning trees of~$G$.
   \item{The dual of its finite-cycle matroid $\mfc(G)$ is its bond matroid~$\mb(G)$, whose circuits are the (finite or infinite) bonds of~$G$.}
   \end{enumerate}
\end{theorem}

\noindent
In Section~\ref{graphduality} we shall extend Theorem~\ref{graphdualslocfin} to a slightly larger class of graphs.

\medbreak

It has turned out that these four cycle- and bond-type matroids are only the extremes of a rich class of matroids associated with the topological circles or the bonds in a locally finite graph~$G$. Indeed, whenever $\Psi$ is a Borel subset of the full set of ends of~$G$, the edge sets of topological circles in the space obtained from $G$ by adding only the ends in $\Psi$ form a matroid. (There exist graphs $G$ with some non-Borel sets $\Psi$ of ends for which this is not the case.) The duals of these matroids have as their circuits the bonds of $G$ that have {\em no\/} end of~$\Psi$ in their closure. See Bowler and Carmesin~\cite{BC13:Determinacy} for details.

\subsection{Matroids describing the duality of planar graphs}\label{graphduality}

Whitney's theorem~\cite{DiestelBook10noEE} says that a finite graph $G$ is planar if and only if the dual of its cycle matroid is {\em graphic\/}, i.e., is the cycle matroid of some other graph.%
   \COMMENT{}
   Our matroids allow us to extend this to infinite graphs, as follows.

Thomassen~\cite{thomassen82} showed that any reasonable notion of duality for infinite graphs requires that these are \emph{finitely separable\/}: that any two vertices can be separated by a finite set of edges. The class of finitely separable graphs is slightly larger than that of locally finite graphs, and just right for duality: while locally finite graphs can have duals that are not locally finite (with respect to any reasonable notion of duality, e.g.\ geometrically in the plane), duals of finitely separable graphs, as defined formally below, are again finitely separable.

Since bonds can be infinite, any adaptation of graph duality to infinite graphs that takes account of all bonds%
   \COMMENT{}
   requires a notion of possibly infinite circuits for graphs: the edge sets which, if the graph is planar, will be the bonds of its dual. The notion that works for finitely separable graphs extends that defined for locally finite graphs in Section~\ref{SecCycleBond}: take as the {\em circuits\/} the edge sets of topological circles in the quotient space $\tilde G$ of $|G|$ obtained by identifying every vertex with all the ends from which it cannot be separated by finitely many edges.%
   \footnote{Equivalently: by finitely many vertices. Another way of obtaining $\tilde G$ is to start not from $|G|$ but directly from~$G$: we simply add only those ends that are not dominated by a vertex in this way, while making rays of the other ends converge to the vertex dominating that end. See~\cite{RDsBanffSurvey,TST} for details.}%
   \COMMENT{}
   (Note that, since $G$ is finitely separable, no two vertices are identified with the same end.) 

As in locally finite graphs, these edge sets are the circuits of a matroid, the {\em topological cycle matroid\/} $\mc(G)$ of~$G$. As before, this is the dual of the finite-bond matroid~$\mfb(G)$. The bonds of $G$, finite or infinite, are also once more the circuits of a matroid, the {\it bond matroid\/}~$\mb(G)$ of~$G$.

\begin{theorem}$\!\!${\rm\cite{BruhnDiestelMatroidsGraphs}}\label{MCinGtilde}
Let $G$ be a finitely separable connected graph.
   \begin{enumerate}[\rm (i)]\itemsep=0pt
   \item The topological cycle matroid of $G$ is the dual of its finite-bond matroid.
   \item The bond matroid of $G$ is the dual of its finite-cycle matroid.
   \end{enumerate}
\end{theorem}

A finitely separable graph $G^*$ is a {\em dual\/} of a finitely separable graph~$G$ with the same edge set if the bonds of $G^*$ 
are precisely the circuits of~$G$, the edge sets of the topological circles in~$\tilde G$. 
It has been shown in~\cite{duality} that, if $G$ is 3-connected, this graph $G^*$ is unique, 3-connected, and has $G$ as its unique dual, so $G^{**}=G$.

By Theorem~\ref{MCinGtilde}, graph duality commutes with matroid duality:%
   \COMMENT{}

\begin{corollary}\label{graphicduality}
If $G$ and $G^*$ are dual finitely separable graphs, then
 $$\mfc^*(G) = \mc(G^*) \quad \text{and} \quad \mfb^*(G) = \mb(G^*).$$
  \vskip-12pt\vskip0pt\qed
\end{corollary}

\medbreak

Call a matroid \emph{topologically graphic} if it is the topological cycle matroid $\mc(G)$ of some graph~$G$, and \emph{finitely graphic} if it is the finite-cycle matroid $\mfc(G)$ of some graph~$G$. We then have the following infinite version of Whitney's theorem:

\begin{theorem}$\!\!${\rm\cite{BruhnDiestelMatroidsGraphs}}\label{Whitney}
The following three assertions are equivalent for a countable finitely separable graph~$G$:
\begin{enumerate}[\rm (i)]\itemsep=0pt
\item $G$ is planar;
\item $\mc^*(G)$ is finitely graphic;
\item $\mfc^*(G)$ is topologically graphic.
\end{enumerate}
\end{theorem}

\goodbreak

\noindent
The graphs witnessing (ii) and (iii) can also be chosen to be finitely separable~\cite{BruhnDiestelMatroidsGraphs}.%
  \COMMENT{}

\medbreak

As before, Corollary~\ref{graphicduality} and Theorem~\ref{Whitney} are but the extreme cases of a more subtle duality of planar graphs, which is reflected by the $\Psi$-matroids indicated at the end of Section~\ref{SecCycleBond}. See~\cite{BC13:Determinacy, DP11:DualTrees} for details. More on infinite graphic matroids can be found in~\cite{BCC:graphic_matroids}.

\subsection{The algebraic cycle matroid of a graph}\label{algebraic}

Another natural matroid in a locally finite graph $G$ is its {\em algebraic cycle matroid\/}: the matroid whose circuits are the {\em elementary algebraic cycles\/} of $G$ (in the sense of infinite simplicial 1-chains with zero boundary), the minimal non-empty edge sets inducing even degrees at all vertices. When $G$ is infinite, these are the edge sets of its (finite) cycles and those of its {\em double rays\/}, its 2-way infinite paths.

\begin{figure}[htbp]
  \centering
  \includegraphics[width=0.7\linewidth]{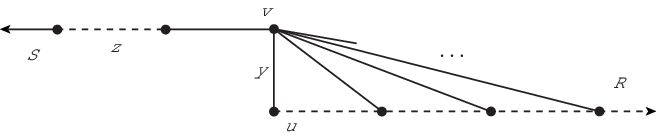}
  \caption{The Bean graph}
  \label{Beangraph}
\end{figure}

The elementary algebraic cycles do not form a matroid in every infinite graph: we shall see in Section~\ref{non-matroids} that they do not satisfy our circuit axioms when $G$ is the {\em Bean graph\/} shown in Figure~\ref{Beangraph}.
However, Higgs proved (for his `B-matroids'; but cf.\ Theorem~\ref{Bmatroids}) that this is essentially the only counterexample:

\begin{theorem}[Higgs~\cite{Higgs69graphs}]\label{HiggsBean}
The elementary algebraic cycles of an infinite graph $G$ are the circuits of a matroid on its edge set $E(G)$ if and only if $G$ contains no subdivision of the Bean graph.
\end{theorem}

\begin{corollary}
The elementary algebraic cycles of any locally finite graph are the circuits of a matroid.\qed
\end{corollary}

The dual of the algebraic cycle matroid of a graph $G$ can also be described: it is the matroid whose circuits are the minimal non-empty cuts dividing the graph into a rayless `small' side and the rest~\cite{BruhnDiestelMatroidsGraphs}.

\subsection{A matroid without finite circuits or cocircuits}

Most of the examples of infinite matroids we have seen so far are either finitary or cofinitary. The algebraic cycle matroids discussed in the last section, however, can have both infinite circuits and infinite cocircuits. The following example is an extreme case, in that its circuits and cocircuits are {\em all\/} infinite:

\begin{example}\label{Tinfty}
The matroid of the 
elementary algebraic cycles in the $\aleph_0$-regular tree~$T_\infty$ has no finite circuit and no finite cocircuit.
\end{example}

\proof
   Clearly, the elementary algebraic cycles of $T_\infty$ are just the edge sets of its double rays. 
Since $T_\infty$ does not contain the Bean graph as a subdivision, they are the circuits of a matroid
$MT_\infty$ on the edge set of~$T_\infty$, by Theorem~\ref{HiggsBean}.%
   \footnote{This can also be seen directly. Checking (C1--3) is easy; see \cite{HenningHabil} for a direct proof of~(CM).}%
   \COMMENT{}

To show that every cocircuit is infinite we borrow Lemma~\ref{circuitcocircuitcap} from Section~\ref{basics}, which says that a circuit and a cocircuit never meet in exactly one element. Since for any finite edge set $F$ in~$T_\infty$ it is easy to find a double ray meeting $F$ in exactly one edge, we deduce that $F$ cannot be a cocircuit.
   \endproof

\subsection{Representability and thin independence}\label{thin}

An important class of finite matroids are the representable matroids~\cite{WhittleRep05}. However, as matroids defined by linear independence are finitary, the dual of an infinite representable matroid will not, except in trivial cases, be representable. Representability thus seems to be a concept too narrow for infinite matroids. The following notion of \emph{thin independence}, which agrees with linear independence when the matroid is finite, leads to an otherwise slightly weaker%
   \COMMENT{}
   notion of representability that is more compatible with duality.

Let $F$ be a field, and let $A$ be some set. We say that a family $\Phi = (\varphi_i)_{i\in I}$ of functions $\varphi_i\colon A\to F$
is \emph{thin\/} if for every $a\in A$ there are only finitely many $i\in I$ with $\varphi_i(a)\neq 0$.
Given such a thin family~$\Phi$ of functions, their pointwise sum $\sum_{i\in I}\varphi_i$ is another $A\to F$ function.%
   \COMMENT{}
   We say that a family~$\Psi$ of $A\to F$ functions, not necessarily thin, is \emph{thinly independent}
if for every thin subfamily $\Phi = (\varphi_i)_{i\in I}$ of $\Psi$ and every corresponding%
   \COMMENT{}
   family $(\lambda_i)_{i\in I}$ of coefficients $\lambda_i\in F$ we have $\sum_{i\in I}\lambda_i\varphi_i=0\in F^A$ only when $\lambda_i=0$ for all $i\in I$.%
   \COMMENT{}

Unlike with linear independence, the thinly independent subfamilies of a given family of $A\to F$ functions do not always form a matroid.%
   \footnote{View the elements of $E=\mathbb F_2^\mathbb N$ as subsets of $\mathbb N$,
and define sets $I:=\{\{1,n\}:n\in\mathbb N\}$ and $I':=\{\{n\}:n\in\mathbb N\}$.
Both $I$ and $I'$ are thinly independent. Moreover, $I'$ is maximally thinly independent
but $I$ is not: $I+\mathbb N$, for instance, is still thinly independent. 
Yet, the only $x\in I'$ for which $I+x$ is thinly independent is $x=\{1\}$,
which is already contained in $I$. Thus, (I3) is violated.}
   But they do if the given family of functions is itself thin:

\begin{theorem} {\rm \cite{BruhnDiestelMatroidsGraphs}}\label{thinsumsmatroid}
If a family $E$ of $A\to F$ functions is thin, then its thinly independent subfamilies are
the independent sets of a matroid on~$E$.%
   \COMMENT{} 
\end{theorem}

\noindent
Afzali and Bowler~\cite{AfzaliBowler12} have shown that the matroids arising as in Theorem~\ref{thinsumsmatroid} are precisely the duals of the matroids that are representable over~$F$ in the usual sense. 

\medbreak

Whenever the thinly independent subfamilies of a family of $A\to F$ functions form a matroid, we call this the \emph{thin-sums matroid} of these functions. We say that a matroid can be \emph{thinly represented} over $F$ if it is isomorphic (in the obvious sense) to such a matroid. For finite matroids, thin representability over a given field is easily seen to coincide with ordinary representability over that field.%
   \COMMENT{}

Many standard infinite matroids, including all the variants of cycle and bond matroids of locally finite graphs~\cite{AfzaliBowler12} (see Sections~\ref{SecCycleBond}--\ref{algebraic})%
   \COMMENT{}
    or of higher-dimensional complexes (Section~\ref{complexes}), are thinly representable. Every matroid that is representable over a field~$F$ in the usual sense is also thinly representable over~$F$~\cite{AfzaliBowler12}, which is not obvious. So thin representability generalizes ordinary representability.%
   \COMMENT{}
   Conversely, every finitary matroid that is thinly representable over~$F$ is also representable over~$F$ (which is easy).%
   \COMMENT{}
   Hence for all finitary matroids, not just for the finite ones, thin representability coincides with ordinary representability, but for infinite (finitary) matroids this coincidence is not the canonical one as for finite matroids.%
   \COMMENT{}

The class of thinly representable matroids is not closed under duality~\cite{BC12:wildmatroids}. However, there is an important subclass that is: the class of `tame' thinly representable matroids (which is also closed under taking minors)~\cite{AfzaliBowler12}.%
   \COMMENT{}
   A matroid, thinly representable or not, is {\em tame\/} if every circuit meets every cocircuit only finitely. Tame matroids are substantially easier to handle than arbitrary matroids, and have more pleasant properties~\cite{BC13:Determinacy}. Forbidden minor characterizations extend readily from finite to tame matroids~\cite{BC12:excluded_matroid_minors}; for example, a tame matroid is thinly representable over $\F_2$ if and only if it does not have $U_{2,4}$ as a minor. The class of tame matroids is closed under taking minors%
   \COMMENT{}
   (as well as, by definition, under duality), so the tame matroids also solve Rado's original problem. See Bowler and Carmesin~\cite{BC12:excluded_matroid_minors, BC13:Determinacy, BC12:wildmatroids, BC13:Ubiquity} for more.

Generalizing matroid representations over fields and over a finite ground set to representations over fuzzy fields and infinite ground sets, Dress defined {\em matroids with coefficients} \cite{Dress86}. There is a fuzzy field over which all finite matroids are representable~\cite{Dress-WenzelGrassmann-Pluecker}.%
  \COMMENT{} 
  A matroid with coefficients, $D$~say, determines a closure operator~$cl_D$ that satisfies (CL1--4), but which need not satisfy~(CLM). In this setup, $D$~also has a {\em dual\/}~$D^*$, which by construction has the property that its circuits intersect those of $D$ finitely. We do not know whether this duality agrees with ours when $cl_D$ and~$cl_{D^*}$ do satisfy~(CLM) and hence define matroids in our sense, i.e., whether then $(E, cl_{D^*})=(E, cl_D)^*$. But there is an example of a matroid with coefficients, $D$~say, for which $cl_D$ satisfies (CLM) but $cl_{D^*}$ does not~\cite{AfzaliBowler12}. Results of Wagowski~\cite{Wag94} imply that, given a matroid $M$ in our sense, there is a matroid with coefficients $D$ such that $M=(E,cl_D)$ and $M^*=(E,cl_{D^*})$ if and only if $M$ is tame.

\subsection{The algebraic cycle matroid of a complex}\label{complexes}

Before we turn to more general complexes, let us show that the algebraic cycle matroid of a graph $G=(V,E)$ can be thinly represented  over~$\mathbb F_2$, for any $G$ for which it is defined (cf.\ Theorem~\ref{HiggsBean}). We represent an edge $e=uv$ by the map $V\to\mathbb F_2$ assigning 1 to both $u$ and~$v$, and 0 to every other vertex. Then a set $F\sub E$ of edges becomes a family $(\varphi_f)_{f\in F}$ of $V\to\mathbb F_2$ functions, not necessarily thin, which is thinly independent if and only if $F$ contains no elementary algebraic cycle.%
   \COMMENT{}

The above example generalizes to higher dimensions. Let $K$ be a locally finite simplicial complex. Let us show that, for each $n\in\N$, the $n$-dimensional cycles in $K$ define a matroid $M_n(K)$ on the set $\Delta_n(K)$ of its $n$-simplices, which is thinly representable over~$\mathbb F_2$.

Formally, we define this matroid as a thin-sums matroid over~$\mathbb F_2$, representing each simplex $\sigma\in\Delta_n(K)$ by its boundary~$\partial\sigma$: this is an $(n-1)$-chain with coefficients in~$\mathbb F_2$, which we think of as a function $\varphi_\sigma\colon \Delta_{n-1}(K)\to\mathbb F_2$.%
   \footnote{In the notation of Section~\ref{thin}, we have $A = \Delta_{n-1}(K)$ and index sets $I\sub\Delta_n(K)$.}
   Thus, formally, our ground set~$E$ is not $\Delta_n(K)$ itself (the intended reading) but the family $(\varphi_\sigma)_{\sigma\in\Delta_n(K)}$.%
   \COMMENT{}
   Since $K$ is locally finite, the sets $F\sub E$ are thin families of such functions. Such a family $F=(\varphi_\sigma)_{\sigma\in\Sigma}$ is thinly independent if and only if it contains no non-trivial $n$-cycle, that is, has no non-empty subfamily $F' = (\varphi_\sigma)_{\sigma\in\Sigma'}$ such that $\partial\psi = 0$ for the corresponding $n$-chain $\psi := \sum_{\sigma\in\Sigma'}\sigma$.

Theorem~\ref{thinsumsmatroid} thus has the following application:

\begin{theorem}{\rm \cite{BruhnDiestelMatroidsGraphs}}
The minimal non-zero $n$-dimensional cycles of a locally finite simplicial complex form the circuits of a matroid.
\end{theorem}

\noindent
Let us call this matroid the {\em $n$-dimensional cycle matroid\/} of the complex~$K$, and denote it by~$M_n(K)$. In general, this is a non-finitary matroid, but by the result of Afzali and Bowler~\cite{AfzaliBowler12} mentioned after Theorem~\ref{thinsumsmatroid} it is always cofinitary.%
   \COMMENT{}

\medbreak

We remark that even for $n=1$ it was not entirely trivial to verify~(CM) for this matroid. For $n>1$ we know of no direct proof. The other essential axioms, such as (C3), (I3) or~(B2), also appear to be hard to verify directly when the complex is infinite.%
   \COMMENT{}

\section{Basic properties}\label{basics}

In this section we prove just enough about infinite matroids $(E,\I)$ to enable us in Section~\ref{eq} to deduce that the various axiom systems given in Section~\ref{axioms} are indeed equivalent. On the way we define duality, deletions and contractions, and show that they behave as for finite matroids. More properties of infinite matroids, especially regarding connectivity, are proved in~\cite{BruhnWollanConInfMatroids}.

Let $M=(E,\I)$ be a fixed matroid, that is, assume throughout this section that $\I$ satisfies 
the independence axioms. Write $\B:=\Imax$ for its set of bases. We start with an observation that can be directly read off the axioms:%
   \COMMENT{}

\begin{itemize}\COMMENT{}
\item[\rm (I3$'$)] For all $I\in\I$ and $I'\in\B$ there is a $B\in\B$ such that $I\sub B\sub I\cup I'$.
\end{itemize}
Indeed, by (IM) there is a maximal independent subset $B$ of ${I\cup I'}$ such that $I\sub B$. Then $B\in\B$, as otherwise we could use (I3) to extend $B$ further into~$I'$ (keeping it independent), contrary to its definition.

\medbreak

Let us establish duality. Define
 $$\B^* := \{\,B^*\sub E\mid E\sm B^*\in\B\,\}\quad \rlap{$\big(= \{\,\Bbar\mid B\in\B\,\}\big)$}$$
and $\I^* := \dcl{\B^*}$.

\begin{theorem}\label{duality}
If $\I$ satisfies the independence axioms, then so does~$\I^*$, with $\B^*$ as its set of bases.
\end{theorem}

\proof
Since $\I$ satisfies (I1) and~(IM), we have $\B^*\ne\es$ and hence (I2) and (I1) for~$\I^*$. Since $\B$ and hence also $\B^*$ is an antichain, we have $\I^{*\rm max} = \B^*$. To prove (I3) for~$\I^*$, let $I^*\in\,\I^*\sm \B^*$ be given, with $I^*\cap B = \es$ for $B\in\B$ say,%
   \COMMENT{}
   and let also $B'\in\B$ be given; our aim is to extend $I^*$ non-trivially into~$\overline{B'}$ while keeping it in~$\I^*$.%
   \COMMENT{}

We first use (I3$'$) to extend the independent set $B'\sm I^*$ into~$B$, to a subset $B''\in\B$ of $(B'\sm I^*)\cup B$. Then $I^*\sub \overline{B''}\in\B^*$, and the inclusion is proper since $I^*\notin\B^*$ by assumption. But
 $$\overline{B''}\sm I^* = \overline{B''\cup I^*} \sub \overline{B'\cup I^*} = \overline{B'}\sm I^*$$
 since $B'\cup I^*\sub B''\cup I^*$. So the extension $\overline{B''}$ of $I^*$ is as desired, completing our proof of (I3)---indeed of its strengthening~(I3$'$)---for~$\I^*$.

It remains to prove that $\I^*$ satisfies~(M). Let $X\sub E$ and $I^*\in \I^*\cap 2^X$ be given. By definition of~$\I^*$, there exists a set $B\in\B$ such that $I^*\cap B = \es$. By~(IM), ${\Xbar}$ has a maximal independent subset~$I$. By~(I3$'$), we can extend $I$ to a subset $B'\in\B$ of $I\cup B\sub \overline{I^*}$.%
   \COMMENT{}

We claim that $X\sm B'$ witnesses (M) for $I^*$ and~$\I^*$, i.e.\ that $X\sm B'$ is maximal among the subsets of $X$ that contain $I^*$ and avoid an element of~$\B$.%
   \COMMENT{}
   Suppose not. Then there is a set $B''\in\B$ such that $B''\cap X\subne B'\cap X$. Then
 $$I' := (B''\cap X)\cup (B'\sm X) \subne B',$$
 so $I'\in\I\sm\B$. We can thus use (I3) to extend $I'$ properly into~$B''$ to a larger independent set~$I''$. But then $I\sub I'\sm X\subne I''\sm X$, contradicting the choice of~$I$.%
   \looseness=-1
   \endproof

Given a matroid $M = (E,\I)$, we call the matroid $M^* := (E,\I^*)$ specified by Theorem~\ref{duality} the {\em dual\/} of~$M$. As usual, we call the bases, circuits, dependent and independent sets of $M^*$ the {\em cobases, cocircuits, codependent\/} and {\em coindependent\/} sets of~$M$.

\medbreak

Next, let us show that our matroids have restrictions defined in the usual way: that, given a set $X\sub E$, the pair $(X,\,\I\cap 2^X)$ is again a matroid. It will be convenient to use the following duality argument in our proof of this fact:

\begin{lemma}\label{cobases}
If $X\sub E$ and $B\in\B$, then $B\cap X$ is maximal in $\I\cap 2^X$ if and only if $\Bbar\cap\Xbar$ is maximal in $\I^*\cap 2^{\Xbar}$.
\end{lemma}

\proof
Suppose first that $B\cap X$ is maximal in $\I\cap 2^X$. If $\Bbar\cap\Xbar$ is not maximal in $\I^*\cap 2^{\Xbar}$, there exists some $B'\in\B$ such that $B'\sm X\subne B\sm X$. Use (I3$'$) to extend $I:= B\cap X$%
   \COMMENT{}
   to a subset $I'\in\B$ of $I\cup B'$. Then $I'\cap X\supsetneq B\cap X$, since $I'$ is not a proper subset of~$B$. This contradicts our initial assumption about~$B$.

The converse implication follows by taking complements.
\endproof

\begin{lemma}\label{restr}
For every set $X\sub E$, the set $\I\cap 2^X$ satisfies~{\rm (I3$'$)}.
\end{lemma}

\proof
Let an independent subset $I$ of $X$ and a maximal independent subset $I'$ of~$X$ be given. Using~(IM) in~$E$,%
   \COMMENT{}
   extend $I'$ to a set $B'\in\B$. Note that $I' = B'\cap X$, by the maximality of~$I'$. By Lemma~\ref{cobases},
$$\text{$\overline{B'}\cap\Xbar$ is maximal in $\I^*\cap 2^{\Xbar}$.}\eqno(*)$$
Use (I3$'$) to extend $I$ into~$B'$, to a subset $B\in\B$ of $I\cup B'$. Then $B\cap\Xbar\sub B'\cap\Xbar$ and hence $\overline{B}\cap\Xbar \supe \overline{B'}\cap\Xbar$. Thus by~$(*)$, the set $\overline{B}\cap\Xbar$ is maximal in $\I^*\cap 2^{\Xbar}$. Applying Lemma~\ref{cobases} backwards, we deduce that $B\cap X$ is maximal in $\I\cap 2^X$. Since
 $$I\sub B\cap X\sub (I\cup B')\cap X = (I\cap X)\cup (B'\cap X) = I\cup I'$$%
 (recall that $I' = B'\cap X$), this completes the proof.
\endproof

\begin{theorem}\label{restrictions}
For every set $X\sub E$, the pair $(X,\,\I\cap 2^X)$ is a matroid.
\end{theorem}

\proof
Axioms (I1), (I2) and (IM) hold for the sets in $\I\cap 2^X$ because they hold for~$\I$. Axiom (I3) for $\I\cap 2^X$ follows from Lemma~\ref{restr}.
\endproof

Given a matroid $M = (E,\I)$ and $X\sub E$, we denote the matroid $(X,\I\cap 2^X)$ as~$M|X$ or as~$M-\Xbar$, and call it the {\em restriction\/} of $M$ {\em to~$X$}, or the {\em minor\/} of $M$ obtained by {\em deleting~$\Xbar$}. Following Oxley~\cite{OxleyBook}, we call
 $$M.X := M/\Xbar := (M^*|X)^*$$
 the {\em contraction\/} of $M$ {\em to~$X$}, or the {\em minor\/} of $M$ obtained by {\em contracting~$\Xbar$}.

\begin{lemma}\label{bases}
The following statements are equivalent for all sets $I\sub X\sub E$:
\begin{enumerate}[\rm (i)]\itemsep=0pt
\item $I$ is a base of $M.X$.\label{basesi}
\item There exists a base $I'$ of $M-X$ such that $I\cup I'\in\B$.\label{basesii}
\item $I\cup I''\in\B$ for every base $I''$ of~$M-X$.\label{basesiii}
\end{enumerate}
\end{lemma}

\proof
(i)~means that $X\sm I$ is a base of $M^*|X$, a~maximal subset of $X$ extending to a base of~$M^*$. (Equivalently, $I$ is minimal with the property that we can extend it to a base of $M$ by adding points of $\Xbar$ only.) By Lemma~\ref{cobases},%
   \COMMENT{}
   this is equivalent to~(ii).

Since $M-X$ is a matroid (Theorem~\ref{restrictions}) it has a base, so (iii) implies~(ii). To prove the converse implication, assume (ii) and let $I''$ be a base of~$M-X$. Use (I3$'$) to extend $I''$ into $B':=I\cup I'$, i.e.\ to find a set $B''\in\B$ such that $I''\sub B''\sub I''\cup B'$. By the minimality of $I$ mentioned in the proof of (i)$\leftrightarrow$(ii), we have $B''\cap X\supe I$, and by the maximality of $I''$ as a base of $M-X$ we have $B''\sm X \sub I''$. In both cases we trivially also have the converse inclusion, so $B'' = I\cup I''$ as desired.
\endproof

\begin{corollary}\label{CoindependentSets}
A set $I\sub X$ is independent in $M.X$ if and only if $I\cup I'\in\I$ for every independent set $I'$ of~$M-X$.
\end{corollary}

\proof
The forward implication follows easily from Lemma~\ref{bases} (i)$\to$(iii).%
   \COMMENT{}

For the backward implication, choose $I'$ as a base of~$M-X$. Use (IM) to extend $I\cup I'\in\I$ to a base $B\in\B$. Then $B\sm X = I'$ by the maximality of~$I'$, so $B\cap X\supe I$ is a base of $M.X$ by (ii)$\to$(i) of Lemma~\ref{bases}.
\endproof

Our next aim is to show the counterpart of (IM) for dependent sets: that inside every dependent set we can find a minimal one, a circuit. For the proof we need another lemma:

\begin{lemma}\label{card}
If bases $B,B'$ satisfy $|B\sm B'| < \infty$, then $|B\sm B'| = |B'\sm B|$.
\end{lemma}

\proof
Suppose not, and choose a counterexample $(B,B')$ with $|B\sm B'|$ minimum. Then $|B\sm B'| < |B'\sm B|$. Pick $x\in B\sm B'$,%
   \COMMENT{}
   and use (I3$'$) to extend $B-x$ to a subset $B''\in\B$ of~$(B-x)\cup B'$. Then $(B'',B')$ is not a counterexample, so the extension $B''\sm (B-x)$ contains at least two elements $y,z$. Now use (I3$'$) to extend $B''-z$ back into~$B$: this yields the base $(B''-z)+x\supe B+y$, which contradicts the maximality of~$B$ as a base.\endproof

\begin{lemma}\label{C0}
Every dependent set contains a circuit.
\end{lemma}

\proof
By Theorem~\ref{restrictions}, it suffices to assume that $E\notin\I$ and find a circuit in~$E$. Pick a base $B\in\B$; this exists by (I1) and~(IM). Then $B\subne E$; pick $z\in E\sm B$. We shall prove that
 $$C := \{\,x\in B+z\mid B+z-x\in\I\,\}\,.$$
 is a circuit. Note that $z\in C$.

We first show that $C$ is dependent. Suppose not, and use (I3$'$) to extend $C$ to a subset $B'\in\B$ of $C\cup B = B+z$. Since $B'\sm B = \{z\}$, we have $|B\sm B'| = 1$ by Lemma~\ref{card}, say $B\sm B' = \{y\}$. But then $B+z-y = B'\in\B$, so $y\in C\sub B'$ by definition of~$C$. This contradicts the definition of~$y$.

$C$ is minimally dependent, since for every $x\in C$ we have $C-x\sub B+z-x\in\I$ by definition of~$C$.
\endproof

Recall that a matroid is called {\em finitary\/} if any set whose finite subsets are independent is also independent.

\begin{corollary}\label{finitary}
A matroid is finitary if and only if every circuit is finite.
\end{corollary}

\proof
A finitary matroid clearly has no infinite circuits.%
   \COMMENT{}
   Conversely, a set whose finite subsets are independent cannot contain a finite circuit. Hence if all circuits are finite it contains no circuit, and is therefore independent by Lemma~\ref{C0}.
\endproof

Let $\cl\colon 2^E\to 2^E$ be the closure operator associated with~$\I$. 

\begin{lemma}\label{FactE}
If $B$ is a maximal independent subset of~$X$, then $\cl(B) = \cl(X)$.
\end{lemma}

\proof
The inclusion $\cl(B)\sub\cl(X)$ is trivial since $B\sub X$; we show the converse. Let $y\in \cl(X)$ be given, witnessed by an independent set $I\sub X$ such that $I+y\notin\I$. By~(IM), we can extend $I$ to a maximal independent subset $B'$ of~$X+y$. Clearly $y\notin B'$, so $B'\sub X$. If $y\in \cl(B)$ we are done. If not then $B+y\in\I$, so $B$ is an independent but not a maximal independent subset of~$X+y$. By Lemma~\ref{restr}, we may use (I3) in $X+y$ to extend $B$ into $B'$ to an independent subset of~$X$ that contains $B$ properly, contradicting its maximality.
\endproof

The following lemma was already used in the proof of Example~\ref{Tinfty}:

\begin{lemma}\label{circuitcocircuitcap}
A circuit and a cocircuit of a matroid never meet in exactly one element.
\end{lemma}

\proof
Let $C$ be a circuit, and $D$ a cocircuit, such that $C\cap D = \{x\}$. As $D-x$ is coindependent, it misses a base~$B$. Apply (I3$'$) to extend the independent set ${C-x}$ to a base $B'\sub (C-x)\cup B$.%
   \vadjust{\penalty-200}
   Since $C$ is dependent and $C-x\sub B'$, we have $x\notin B'$. Hence $D\cap B'=\es$, contradicting our assumption that $D$ is codependent.
\endproof

We still have to show that the relative rank function $r$ associated with~$\I$ is well defined:

\begin{lemma}\label{rwelldefined}
Given $B\sub A\sub E$, there exist maximal elements $J$ of~$\I\cap 2^B$ and $I$ of~$\I\cap 2^A$ such that $J\sub I$. All such sets $J$ and $I$ satisfy $|I\sm J| = r(A|B)$.
\end{lemma}

\proof
Let $J$ be an arbitrary maximal element of~$\I\cap 2^B$; it exists by~(IM). Use (IM) again to extend $J$ to a maximal element $I$ of~$\I\cap 2^A$. To show that $|I\sm J| = r(A|B)$,%
   \COMMENT{}
consider any pair $J'\sub I'$ of independent subsets of $A$ such that $J'$ is maximally independent in~$B$. We have to show that $|I'\sm J'|\le |I\sm J|$.

We may assume that $I'$ is maximal in~$\I\cap 2^A$, and that $J'=J$: if not, we could use Lemma~\ref{bases} (\ref{basesii})$\Leftrightarrow$(\ref{basesiii}) in $M|A$ to replace $J'$ with~$J$  in~$I'$ without affecting $I'\sm J'$. If both $I\sm J$ and $I'\sm J$ are infinite, we have $|I'\sm J'| = |I\sm J|$ as desired.%
   \footnote{Recall our convention that we do not distinguish between infinite cardinalities.}
   If one of them is finite, then $|I'\sm J'| = |I'\sm J| = |I\sm J|$ by Lemma~\ref{card} applied in~$M|A$.

\endproof

\begin{lemma}\label{rminor} 
Let $Y\subseteq X\subseteq E$, and let $r'$ be the relative rank function associated with $(M|X)/Y$. Then for any $A, B$ such that $Y\subseteq B\subseteq A\subseteq X$, we have 
$$r'(A\sm Y\,|\,B\sm Y)=r(A|B).$$
\end{lemma}

\proof By (IM) there exist maximal elements $K$ of $\I\cap 2^Y$ and $J\supe K$ of $\I\cap 2^B$ and $I\supe J$ of $\I\cap 2^A$. Then $|I\sm J| = r(A|B)$ by Lemma \ref{rwelldefined}. But $I\sm J = (I\sm K)\sm (J\sm K)$. As $I\sm K$ and $J\sm K$ are bases of $(M|A)/Y$ and $(M|B)/Y$, respectively (Lemma~\ref{bases}), another application of Lemma \ref{rwelldefined} yields
 $$r(A|B) = |I\sm J| = |(I\sm K)\sm (J\sm K)| = r'(A\sm Y\,|\, B\sm Y)$$
as desired.
\endproof

\section{Equivalence of the axiom systems}\label{eq}

In this section we prove that our axiom systems are equivalent. In our use of the terms `dependent', `independent', `base', `circuit' and `closure' we stick to their definitions as given in Section~\ref{independenceaxioms}, referring to a set system $\I$ known or assumed to satisfy the independence axioms. When we do not assume this, as will often be the case in this section, we shall use unambiguous other terms defined in the context of the axioms assumed, such as `maximal $\C$-independent set'.

\begin{theorem}\label{Beq}\
\begin{enumerate}[\rm (i)]\itemsep=0pt
   \item If a set $\I\sub 2^E$ satisfies the independence axioms, then the set $\B$ of bases satisfies the base axioms with $\I$ as the set of $\B$-independent sets.
   \item If a set $\B\sub 2^E$ satisfies the base axioms, then the set $\I$ of $\B$-independent sets satisfies the independence axioms with $\B$ as the set of bases.
\end{enumerate}
\end{theorem}

\goodbreak

\proof
(i) Let $\I$ satisfy the independence axioms. Applying (IM) with $X:= E$, we see that every set in $\I$ extends to a set in~$\Imax$. Hence (I1) implies~(B1), and $\I=\dcl{\Imax}$; in particular, (IM) implies~(BM).%
   \COMMENT{}

To prove~(B2), let $B_1,B_2\in\B:=\Imax$ and $x\in B_1\sm B_2$ be given. Applying (I3) with $I:= B_1-x$ and $I':= B_2$, we find an element $y\in B_2\sm B_1$ such that $B:= (B_1 - x) + y\in \I$. We have us show that $B\in\Imax$. If $B\notin\Imax$, we can apply (I3) with $I:= B$ and $I':= B_1$ to extend $B$ into~$B_1$ to a set $B'\in\I$. But $B_1\sm B = \{x\}$, so this means that $B_1\subne B'\in\I$, as $y\in B'\sm B_1$. This contradicts our assumption that $B_1\in\Imax$.

(ii) Let $\B$ satisfy the base axioms, and let $\I:=\dcl\B$. Then (B1) implies~(I1), (I2)~is trivial, and (BM) trivially implies~(IM). Since by (B2) no set in $\B$ contains another, we also have $\B = \Imax$.

To prove~(I3), let $I\in\I\sm\B$ and $I'\in\B$ be given. Use (IM) with $X:= E$ to extend $I$ to a set $B\in\B=\Imax$, and pick $x\in B\sm I$. If $x\in I'$, then $I + x\in\I$ is as desired. If not, we can use (B2) with $B_1:= B$ and $B_2 := I'$ to find $y\in I'\sm B$ such that $(B-x)+y\in\B$. As $I\sub B-x$ this yields $I+y\in\I$, as required for~(I3).
   \endproof

\begin{theorem}\label{CLeq}\
\begin{enumerate}[\rm (i)]\itemsep=0pt
   \item If a set $\I\sub 2^E$ satisfies the independence axioms, then the associated closure operator \cl\ satisfies the closure axioms with $\I$ as the set of $\cl$-inde\-pen\-dent sets.
   \item If a function $\cl\colon 2^E\to 2^E$ satisfies the closure axioms, then the set $\I$ of $\cl$-independent sets satisfies the independence axioms with \cl\ as the associated closure operator.
\end{enumerate}
\end{theorem}

\proof
(i) Let $\I$ satisfy the independence axioms, and let $\cl$ be the associated closure operator. Then (CL1) and (CL2) hold trivially. By~(I2), every set in $\I$ is $\cl$-independent. Conversely, a $\cl$-independent set $X$ lies in~$\I$: if not, then by (IM) it has a maximal independent subset~$I\subne X$, and every $x\in X\sm I$ satisfies $x\in\cl(I)$, contradicting the $\cl$-independence of~$X$. Hence the $\cl$-independent sets are precisely those in~$\I$, and (IM) implies~(CLM).

To prove~(CL3), let $X\sub E$ be given. By~(IM), $X$~has a maximal independent subset~$B$. By Lemma~\ref{FactE}, $B$~is maximally independent also in $\cl(X) = \cl(B)$.%
   \COMMENT{}
   By Lemma~\ref{FactE} applied to $B$ in~$\cl(X)$ this implies $\cl(B) = \cl(\cl(X))$, yielding $\cl(X) = \cl(\cl(X))$ in total.

Let finally $Z$, $x$ and~$y$ be given for the proof of~(CL4). As $y\in\cl(Z+x)\sm\cl(Z)$, there is an independent set $I\sub Z+x$ such that $I+y$ is dependent and $(I-x)+y\in\I$. As $I-x\sub Z$, this also witnesses that $x\in\cl(Z+y)$.

(ii) Let $\cl\colon 2^E\to 2^E$ satisfy the closure axioms, and let $\I$ be the set of $\cl$-independent sets. Then $\I$ satisfies (I1) and (I2) trivially, and (IM) is just a restatement of~(CLM).

For the remainder of our proof we shall need show the following fact:
\begin{equation}\tag{\ensuremath{\ast}}\label{star}
  \begin{minipage}[c]{0.8\textwidth}
    Whenever a set $Z\sub E$ is $\cl$-independent but $Z+x$ is not\\ (for some $x\in E$), we have $x\in\cl(Z)$.
  \end{minipage}\ignorespacesafterend
\end{equation}
Indeed, by assumption we have $x\notin Z$, and some $y\in Z+x$ lies in the closure of the other elements of~$Z+x$. If $y=x$, then $x=y\in\cl(Z)$ as claimed. If $y\ne x$ then $y\in Z$, so $y\notin\cl(Z-y)$ since $Z$ is $\cl$-independent. Hence $x\in\cl(Z)$ by~(CL4).%
   \COMMENT{}

To prove~(I3), let $I\in\I\sm\Imax$ and $I'\in\Imax$ be given. Use (CLM) to extend $I$ to a maximal element $B$ of $\I\cap 2^{I\cup I'}$. We shall prove that $B$ is maximal in all of~$\I$; then $B\sm I\ne\es$, and any $x\in B\sm I$ proves~(I3).%
   \COMMENT{}

To show that $B\in\Imax$, consider any $z\in E\sm B$. Then $z\in\cl(I')$: trivially if $z\in I'$, or by \eqref{star} and $I'\in\Imax$ if $z\notin I'$. Similarly, the maximality of $B$ in $\I\cap 2^{I\cup I'}$ implies by~\eqref{star} that $I'\sub\cl(B)$. Hence $z\in\cl(I')\sub\cl(\cl(B)) = \cl(B)$ by (CL2) and~(CL3). As $z\notin B$, this means that $B+z\notin I$ as desired.

It remains to show that $\cl$ coincides with the closure operator $\cl'$ associated with~$\I$, i.e.\ that $\cl(X) = \cl'(X)$ for every $X\sub E$. To show that $\cl(X)\sub\cl'(X)$, consider any $x\in\cl(X)$. If $x\in X$ then $x\in\cl'(X)$, so assume that $x\notin X$. Our assumption of $x\in\cl(X)$ now means that $X+x$ is \cl-dependent, that $X+x\notin\I$. By~(CLM), $X$~has a maximal \cl-independent subset~$I$. Then $X\sub\cl(I)$ by~$(*)$,%
   \COMMENT{}
   so $x\in\cl(X)\sub\cl(\cl(I)) = \cl(I)$ by (CL2) and~(CL3), showing that $x\in\cl'(X)$.%
   \COMMENT{}

The converse inclusion, $\cl'(X)\sub\cl(X)$, follows easily from~$(*)$.%
   \COMMENT{}
   \endproof

\begin{theorem}\label{Ceq}\ \vskip-6pt\vskip-6pt
\begin{enumerate}[\rm (i)]\itemsep=0pt
   \item If a set $\I\sub 2^E$ satisfies the independence axioms, then the set $\C$ of circuits satisfies the circuit axioms with $\I$ as the set of $\C$-inde\-pen\-dent sets.
   \item If a set $\C\sub 2^E$ satisfies the circuit axioms, then the set $\I$ of $\C$-independent sets satisfies the independence axioms with $\C$ as the set of circuits.
\end{enumerate}
\end{theorem}

\proof
(i) Let $\I$ satisfy the independence axioms, let $\C$ be the corresponding set of circuits, and let $\cl$ be the closure operator associated with~$\I$. (I1)~implies~(C1), and (C2) holds by definition of~$\C$. By (I2) and Lemma~\ref{C0}, the $\C$-independent sets are precisely those in~$\I$, so (IM) implies~(CM).

To prove~(C3), let $X\sub C\in\C$ and $(C_x\mid x\in X)$ and $z$ be given as stated. Let
  $$Y:= \Big(C\cup\bigcup_{x\in X} C_x\Big) \sm (X+z)\,.$$
  For every $x\in X$ we have $x\in \cl(C_x - x)$ and $(C_x - x)\cap (X+z) = \es$, so
  $$X\ \sub\ \cl\Big(\bigcup_{x\in X} (C_x - x)\sm (X+z)\Big)\ \sub\ \cl(Y).$$
  Hence
  $$C-z = (C\sm (X+z)) \cup X \sub Y\cup\cl(Y) = \cl(Y)$$
  and therefore
  $$z\in \cl (C-z)\sub \cl(\cl(Y)) = \cl(Y)$$
  by Theorem~\ref{CLeq}~(i). So $Y$ has an independent subset $I$ such that $I+z$ is dependent. By Lemma~\ref{C0}, $I+z$~contains a circuit, which clearly contains~$z$.

(ii) Let $\C$ satisfy the circuit axioms, and let $\I$ be the set of $\C$-independent sets. Then (I1) and~(I2) hold trivially,%
   \COMMENT{}
  and (IM) is just a restatement of~(CLM). By~(C2), no element of~$\C$ contains another, so $\C$~is the set of circuits.%
   \COMMENT{}

To prove~(I3), let $I\in\I\sm\Imax$ and $I'\in\Imax$ be given. Use (IM) with $X:= E$ to extend $I$ to a set $B\in\Imax$, and pick $z\in B\sm I$. If $z\in  I'$, then $x:= z$ is as required for~(I3).%
   \COMMENT{}
   If $z\notin I'$, then $I'+z$ contains a set $C\in\C$.%
   \COMMENT{}
   We wish to apply (C3) with $X:= C\sm B$ to obtain a contradiction.%
   \COMMENT{}
   Note that $X\sub I'\sm I$, since $I+z \sub B$. For each $x\in X$ we may assume that $I+x$ contains a set $C_x\in\C$, since otherwise $I+x$ witnesses~(I3). Then $z\notin I+x\supe C_x$ for all $x\in X$, so by (C3) there is a set $C'\in\C$ such that $C'\sub \big(C\cup\bigcup_{x\in X} C_x\big)\sm X.$ As $C_x\sm X\sub I\sub B$ for every~$x$, and $C\sm X\sub B$ by definition of~$X$, this implies that $C'\sub B\in\I$, a contradiction.
\endproof

We remarked in the introduction that, traditionally, infinite (finitary) matroids were defined by specifying that the finite sets in their collection $\I$ of independent sets should satisfy (I1)--(I3), and that the infinite sets in $\I$ were determined by~(I4), i.e., by taking all sets whose finite subsets were known to be in~$\I$. Using the circuit axioms, we can now prove easily that this does in fact define a matroid in our sense, i.e., that (IM) is true and (I3) also holds for infinite sets $I,I'\in\I$:

\begin{corollary}\label{OldFinitaryAreMatroids}
Let $\I\sub 2^E$ satisfy {\rm(I1), (I2)} and~{\rm(I4)},%
   \COMMENT{}
   and assume that the finite sets $I,I'\in\I$ satisfy the usual finite augmentation axiom~{\rm(I3)$_{\rm fin}$}.%
   \COMMENT{}
   Then $\I$ is the set of independent sets of a matroid.
\end{corollary}

\proof
Let $\C$ be the set of all minimal sets in~$2^E\sm\I$. These satisfy (C1) and~(C2), and by~(I4) they are finite. Our assumption of (I1)--(I3)$_{\rm fin}$ for the finite sets in~$\I$ therefore implies (C3) for~$\C$: given $C$ and $\{C_x\mid x\in X\}$ as in~(C3), the set $Y:= C\cup\bigcup_{x\in X} C_x$ is finite, so $\I\cap 2^Y$ is the set of independent sets of a matroid on~$Y$. Its circuits, which are precisely the sets in~$\C\cap 2^Y$,%
   \COMMENT{}
   satisfy the strong elimination axiom, and hence also our axiom~(C3) (induction on~$|X|$). Finally, (CM) follows by Zorn's Lemma. So $\C$ is the collection of circuits of a matroid.

For this to imply the assertion by Theorem~\ref{Ceq}\,(ii), we need that $\I$ contains precisely the $\C$-independent sets, those that have no subset in~$\C$. The sets in $\I$ are $\C$-independent by (I2) and the definition of~$\C$. If a set $D\sub E$ is not in~$\I$, it has a finite subset $F$ not in~$\I$, by~(I4), and hence a minimal such subset~$F$. Then $F\in\C$, so $D$ is not $\C$-independent.
\endproof

\begin{theorem}\label{Req}\ \vskip-6pt\vskip-6pt
\begin{enumerate}[\rm (i)]\itemsep=0pt
   \item If a set $\I\sub 2^E$ satisfies the independence axioms, then the associated relative rank function $r$ satisfies the rank axioms with $\I$ as the set of $r$-inde\-pen\-dent sets.
   \item If a function  $r\colon (2^E\times 2^E)_\subseteq\to \N\cup\{\infty\}$ satisfies the rank axioms, then the set $\I$ of $r$-independent sets satisfies the independence axioms with $r$ as the associated relative rank function.
\end{enumerate}
\end{theorem}

\proof
(i) Let  $\I\sub 2^E$ satisfy the independence axioms and let $r$ be the associated relative rank function. (R1) follows directly from the definition of $r$. We next show (R2), that $r(A\,|\,A\cap B)\geq r(A\cup B\,|\,B)$ for any $A, B\subseteq E$. By Lemma~\ref{rminor}, we may assume that $A\cap B=\emptyset$. By Lemma~\ref{rwelldefined}, there is a maximal set $J$ in $\I\cap 2^B$ and a maximal set $I\in \I\cap 2^{A\cup B}$ such that $J\subseteq I$ and $r(A\cup B\,|\,B)=|I\sm J|$. Then $I\sm J\in \I\cap 2^A$ by (I2), and hence $r(A\,|\,A\cap B)=r(A|\emptyset)\geq |I\sm J|=r(A\cup B\,|\, B)$, as required. When showing (R3), we may assume that $C=\emptyset$ by Lemma~\ref{rminor}. By Lemma~\ref{rwelldefined}, there is a maximal set $J$ in $\I\cap 2^B$ and a maximal set $I$ in $\I\cap 2^A$ so that $J\subseteq I$ and $r(A|B)=|I\sm J|$. Then $r(A|C)=|I|$ and $r(B|C)=|J|$ by definition of $r$, and (R3) follows. 

To prove (R4), consider a family $(A_\gamma)$ and a $B$ such that $B\sub A_\gamma\sub E$ for all~$\gamma$, and let $A:= \bigcup_\gamma A_\gamma$. Suppose $r(A|B)>0$. By Lemma~\ref{rwelldefined}, there is a maximal set $J$ in $\I\cap 2^B$ and a maximal set $I\in \I\cap 2^A$ so that $J\subseteq I$ and $r(A|B)=|I\sm J|$. Then $I\sm J\neq \emptyset$. As $I\subseteq A$, we have $(A_\gamma\cap I)\sm J\neq \emptyset$ and hence $r(A_\gamma|B)\geq |I\cap A_\gamma\sm B|>0$ for some $\gamma$, as required.

We next show for all $I\subseteq E$ that $I\in \I$ if and only if $I$ is $r$-independent. If $I\in \I$, then $r(I\,|\,I-x)>0$ for any $x\in I$ by definition of $r$, so $I$ is $r$-independent. Conversely, if $I\not\in \I$, there exists a maximal element  $J$ of $\I\cap 2^I$. Then $J\subne I$, and $r(I\,|\,I-x)=0$ for any $x\in I\sm J$, proving that $I$ is not $r$-independent. 

As the set of $r$-independent sets equals $\I$, (RM) follows from (IM).

(ii) Assume that $r\colon (2^E\times 2^E)_\subseteq\to \N\cup\{\infty\}$ satisfies the rank axioms, and let $\I$ be the set of $r$-independent sets. Then (I1) holds as the condition for $r$-independence is vacuously satisfied by the empty set. To prove (I2), consider an element $I\in \I$ and some $J\subseteq I$. If $J\not \in \I$, then $r(J\,|\,J-x)=0$ for some $x\in J$. Taking $A=J$ and $B=I-x$ in (R2), we have $r(J\,|\,J-x)\geq r(I\,|\,I-x)$. So $r(I\,|\,I-x)=0$, and $I\not\in \I$, a contradiction. 

Before we show (I3), we make two claims. First, for all $I\in\I$ and $x\in E\sm I$,
  $$I+x\in \I\Longleftrightarrow r(I+x\,|\,I)>0.\eqno(*)$$
The forward implication is immediate from the definition of~$\I$. For the converse implication, assume that $I+x\not\in \I$. Then there is a $y$ such that $r(I+x\,|\,I+x-y)=0$. Then by (R1) and (R3) we have
 $$r(I+x\,|\,I-y)=r(I+x\,|\,I+x-y)+r(I+x-y\,|\,I-y)\leq 1.$$
Applying (R3) again, we obtain
 $$1\ge r(I+x\,|\,I-y)=r(I+x\,|\,I)+r(I\,|\,I-y)\ge r(I\,|\,I-y)\ge 1$$
since $I\in\I$, so $r(I+x\,|\,I)\le 0$ as required.

The second claim is that, for all $X\sub E$ and $I\in \I\cap 2^X$,
$$I\text{ is maximal in } \I\cap 2^X\Longleftrightarrow r(X|I)=0.\eqno(**)$$
Indeed, if $I$ is a maximal element of $\I\cap 2^X$, then $r(I+x\,|\,I)=0$ for all $x\in X\sm I$, by~$(*)$. Taking $A_x:=I+x$ for all $x\in X\sm I$ and $B=I$ in (R4), we find that $r(X|I)=0$. Conversely, suppose that $I$ is not maximal in $\I\cap 2^X$. Then $r(I+x\,|\,I)>0$ for some $x\in X\sm I$, and $r(X|I)=r(X\,|\,I+x)+r(I+x\,|\,I)>0$ by (R3).

We now show (I3). Consider an $I\in\I\sm\Imax$ and an $I'\in\Imax$. By $(**)$ applied with $X=E$, we have $r(E|I)>0$ and $r(E|I')=0$. Applying (R3) twice,%
   \COMMENT{}
   we deduce
 $$0 < r(E\,|\,I) = r(E\,|\,I\cup I') + r(I\cup I'\,|\,I)\le r(E\,|\,I') + r(I\cup I'\,|\,I)
    = r(I\cup I'\,|\,I)\,. $$
Then $(**)$ applied with $X=I\cup I'$ yields that $I$ is not maximal in $\I\cap 2^{I\cup I'}$. Hence there is a set $I''\in  \I\cap 2^{I\cup I'}$ containing $I$ properly. Let $x\in I''\sm I$. Then by (I2), we have  $I+x\in \I$, as required.

(RM) states that (IM) holds for $\I$. This completes our proof that $\I$ satisfies the independence axioms. It remains to show that $r$ is the rank function associated with $\I$.\vadjust{\penalty-200} Let $r'$ be the rank function associated with $\I$, 
and consider $A\supseteq B$. By Lemma~\ref{rwelldefined}, there is a maximal element $I$ of $\I\cap 2^A$ and a maximal element $J$ of $\I\cap 2^B$ so that $J\subseteq I$ and $r'(A|B)=|I\sm J|$. As $I\in \I$, we have $r(I\,|\,I-x)=1$ for any $x\in I$, and $I-x\in \I$ by (I2). Hence, inductively if $I\sm J$ is finite, or by (R3) if it is infinite,%
   \COMMENT{}
   $r(I|J)=|I\sm J|$. Moreover, $r(A|I)=0$ and $r(B|J)=0$ by the maximality of $I$ resp. $J$. Again by~(R3), we deduce that
$$r(A|B) = r(A|B) + r(B|J) = r(A|J) = r(A|I) + r(I|J) = r(I|J)=|I\sm J|=r'(A|B),$$
as required. 
\endproof

\section{Alternative axiom systems and historical links}\label{alt}

In the late 1960s and early '70s, a number of researchers---including Bean, Higgs, Klee, Minty and Las Vergnas---responded to Rado's~\cite{Rado66matroids} challenge to develop a theory of non-finitary infinite matroids that would allow for the kind of duality known from finite matroids.%
   \COMMENT{}
   This resulted in a flurry of related but not easily compatible proposals of how such structures might be defined, of which Higgs's {\em B-matroids\/} were but one among many.%
   \footnote{Higgs himself studied various notions in parallel, including `C-matroids', `transitive spaces', `finitely transitive spaces', `dually transitive spaces', `exchange spaces' and `dually exchange spaces'---as well as two more general structures with duality that he calls `spaces' and `matroids'.}

It was only several years later that Oxley clarified the situation in two ways: he found a simple set of axioms that characterized Higgs's B-matroids~\cite{Oxley78}, and he showed that any theory of finite or infinite matroids with notions of duality and minors that defaulted to the existing finite notions when the structure was finite would have as its models some subclass of those B-matroids. In particular, the models of our matroid axioms proposed in Section~\ref{axioms} must be B-matroids. We shall prove in this section that they are all the B-matroids. Recall, however, that the tame matroids we introduced in Section~\ref{thin} form a smaller class of matroids that is also closed under duality and minors.

Our proof that infinite matroids, as introduced in this paper, are precisely the B-matroids in the sense of Higgs builds on Oxley's axiomatization of the latter: we show that our axiom systems from Section~\ref{axioms} are equivalent to Oxley's axioms for B-matroids. These are of `mixed type': they can be stated either in terms of independent sets or in terms of bases, but each version contains elements of the other. In one flavour, they are the four statements in Theorem~\ref{Bmatroids} below, with (IM) rephrased to fit our terminology:

\begin{theorem}\label{Bmatroids}
A set $\I\sub 2^E$ satisfies the independence axioms if and only if it satisfies the following four statements:
\begin{itemize}\itemsep=0pt
\item[\rm (I1)] $\emptyset\in\mathcal I$.
\item[\rm (I2)] $\I$ is closed under taking subsets.
\item[\rm (IB)] Whenever $X\subseteq E$, the sets $I_1, I_2\sub X$ are maximal elements of $\I\cap 2^X$, and $x\in I_1\sm I_2$, there exists an element $y\in I_2\sm I_1$ such that
$(I_1-x)+y$ is a maximal element of $\I\cap 2^X$.
\item[\rm (IM)] $\I$ satisfies~{\rm(M).}
\end{itemize}
\end{theorem}

\goodbreak

\proof
If $\I$ satisfies the independence axioms, then in particular it satisfies (I1), (I2) and~(IM). Statement (IB) is the base exchange axiom for restrictions, so it holds by Theorems \ref{restrictions} and~\ref{Beq}.%
   \COMMENT{}

Conversely, if $\I$ satisfies the above four statements, then $\Imax$ satisfies the base axioms: (B1)~follows from (I1) and~(IM); (B2)~is the case $X=E$ of~(IB); and (BM) follows from (IM) and~(I2), since these imply that $\I = \dcl{\Imax}$.
\endproof

Given the `exchange' nature of axiom~(IB), it may seem that the four statements above are better rephrased in terms of bases. And indeed, Oxley~\cite{Oxley78}%
   \COMMENT{}
   notes such a translation: a set $\B\sub 2^E$ is the set of bases of a B-matroid if and only if it satisfies (B1), (BM), and~(IB) with $\I:=\dcl{\B}$. These, then, differ from our base axioms only in that they require our exchange axiom (B2) explicitly for all restrictions to subsets $X$ of~$E$. This strengthening makes it necessary to invoke a notion of independent sets, since the `bases' of $M|X$ for which (IB) says that (B2) should hold are defined as the maximal subsets of $X$ in~$\dcl{\B}$. Thus, whichever way we choose to present these axioms, the presentation will involve both elements of base exchange and of independence. Divorcing these into separate sets of independence and base axioms, as we have done in Section~\ref{axioms}, made it necessary to prove that requiring (B2) for all restrictions is in fact redundant in the presence of~(BM)---which we did in our Theorems \ref{restrictions} and~\ref{Beq}.

\medbreak

As a common feature, all our axiom systems so far have included the explicit requirement (M) that every independent set extends to a maximal one---not only in the whole matroid but inside any given $X\sub E$. This is a strong statement, and not always easy to verify in practice.%
   \COMMENT{}
   We therefore tried to replace it with weaker axioms, such as one requiring merely that every set $X\sub E$ must have {\em some} maximal independent subset.

We succeeded in doing this for the independence, the base, and the rank axioms, at the expense of having to strengthen the other axioms a little (see below). For the circuit and the closure axioms we found no natural strengthening that would allow a similar substantial weakening of the (M) axiom.

Let us rephrase the independence axioms first:

\begin{theorem}\label{Ialtdel}
A set $\I\sub 2^E$ satisfies the independence axioms if and only if it satisfies the following three statements:
\begin{itemize}\itemsep=0pt
\item[\rm (I1$'$)] Every set $X\sub E$ has a subset that is maximal in $\I\cap 2^X$.\COMMENT{}\<{(I1d)}
\item[\rm (I2)] $\I$ is closed under taking subsets.
\item[\rm (I3$'$)]
   For\<{(I3c)} all $I\in\I$ and $I'\in\Imax$ there is a $B\in\Imax$ such that $I\sub B\sub I\cup I'$.
\end{itemize}
\end{theorem}

\proof
 Suppose first that $\I$ satisfies the independence axioms: statements (I1), (I2), (I3) and~(IM) from Section~\ref{axioms}. At the start of Section~\ref{basics} we proved that these imply~(I3$'$), and (I1$'$) follows from (IM) with $I:=\es$.

Conversely, assume that $\I$ satisfies (I1$'$), (I2) and~(I3$'$). Axiom (I1) follows from (I1$'$) and~(I2), and (I3) is a weakening of~(I3$'$). To prove~(IM), we begin by re-proving the statement of Lemma~\ref{restr} in Section~\ref{basics}, replacing the use of (IM) in that proof with suitable applications of (I1$'$) and~(I3$'$). By~(I1$'$), the assertion of Lemma~\ref{restr} will then imply~(IM).

We begin by copying the first two paragraphs of the proof of Theorem~\ref{duality}, to show that~$\I^*$, defined as before Theorem~\ref{duality}, satisfies~(I3$'$). (The proof there assumes that the given set $I^*$ is not in~$\B^*$; but if it is, there is nothing to show since $\B^* = \I^{*\rm max}$.) Next, we establish the assertion of Lemma~\ref{cobases} by copying its proof; this uses (I3$'$) for both $\I$ and~$\I^*$,%
   \COMMENT{}
   but it does not use~(IM). Finally, we copy the proof of Lemma~\ref{restr} itself. This proof uses (IM) for $X=E$ in the second line. Instead, we use (I1$'$) with $X:=E$ to find some set $\hat B\in\B$, and then apply (I3$'$) to extend the given set $I'$ into $\hat B$ to the desired set $B\in\B$ (where $I'\sub B\sub I'\cup\hat B$).
   \endproof

We remark that (I1$'$) is weakest possible with the property of completing (I2) and (I3$'$) to a full set of independence axioms. Indeed, since the set of finite subsets of an infinite set satisfies (I2) and (I3$'$) but does not define a matroid, we need to require the existence of a maximal set at least in all of~$E$. Moreover, we want restrictions $M|X$ of a matroid $M$ to be matroids, but the existence of maximal independent sets is not hereditary even in the presence of (I2) and~(I3$'$); see Example~\ref{example2} in Section~\ref{non-matroids}. We thus have to require (I1$'$) as stated.

However, there is an interesting alternative to (I1$'$), which we mention without proof.%
   \COMMENT{}
   Rather than requiring that in every restriction there is a maximal independent set, we may instead prescribe this for every contraction (cf.\ Corollary~\ref{CoindependentSets}): Theorem~\ref{Ialtdel} remains valid if we replace its statement (I1$'$) with

\begin{itemize}
\item[\rm (I1$''$)] Every set $X\sub E$ has a maximal subset $I$ such that $I\cup I'\in\I$ for every $I'\in \I\cap 2^\Xbar$.
\end{itemize}

Next, an alternative set of base axioms. Unlike~(B2), the alternative exchange axiom (B2$'$) does not imply that $\B$ is an antichain, so we have to add this as a new requirement~(B0):

\begin{theorem}\label{Balt}
A set $\B\sub 2^E$ satisfies the base axioms if and only if $\B$ satisfies the following three statements:\vskip-3pt\vskip-3pt
\begin{itemize}\itemsep=0pt
\item[\rm (B0)] No element of $\B$ is a subset of another.\COMMENT{}
\item[\rm (B1$'$)] For every $X\sub E$ there is a $B\in\B$ such that $B\cap X$ is maximal in $\dcl{\B}\cap 2^X$.\<{(B1d)}
\item[\rm (B2$'$)] Whenever\<{(B3a)} $B_1,B_2\in\B$ and $F_1\sub B_1\sm B_2$, there exists $F_2\sub B_2\sm B_1$ such that $(B_1\sm F_1)\cup F_2\in\B$.\COMMENT{}
\end{itemize}
\end{theorem}

\proof
Suppose first that $\B$ satisfies the base axioms: statements (B1), (B2) and~(BM) from Section~\ref{axioms}. (B2)~implies~(B0), and (BM) implies~(B1$'$). To prove~(B2$'$), let $B_1$, $B_2$ and $F_1$ be given as stated. Use (BM) to extend $I:= B_1\sm F_1$ to a maximal subset $B$ of $X:= I\cup B_2$ in~$\dcl{\B}$. We show that $B\in\B$; then (B2$'$) holds with $F_2:= B\sm I$.%
   \COMMENT{}

Suppose $B\notin\B$. Since $B\in\dcl{\B}$, there exists a set $B'\in\B$ such that $B\subne B'$. Note that $B'\sm B\sub \overline{I\cup B_2}$, by the maximality of~$B$ in $\dcl{\B}\cap 2^{I\cup B_2}$. Pick $x\in B'\sm B$. Applying (B2) to $B'$ and~$B_2$, we find an element $y\in B_2\sm B' =  B_2\sm B$ such that $(B'-x)+y\in\B$. This contradicts the maximality of $B$ in $\dcl{\B}\cap 2^{I\cup B_2}$.

Conversely, assume that $\B$ satisfies (B0), (B1$'$) and~(B2$'$). We shall prove that $\I:=\dcl{\B}$ satisfies statements (I1$'$), (I2) and~(I3$'$) from Theorem~\ref{Ialtdel}. Then, by that theorem, $\I$~satisfies the independence axioms. By Theorem~\ref{Beq}, then, $\Imax$~will satisfy the base axioms and $\I = \dcl{\Imax}$.%
   \COMMENT{}
   Hence $\dcl{\B}=\I=\dcl{\Imax}$, which implies $\B=\Imax$ since $\Imax$ and $\B$ are both antichains (by definition and by~(B0), respectively). So $\B$ satisfies the base axioms,%
   \COMMENT{}
   as desired.%
   \COMMENT{}

So let us prove that $\I:=\dcl{\B}$ satisfies (I1$'$), (I2) and~(I3$'$). Statement (I1$'$) is just (B1$'$) for $\I=\dcl{\B}$. Statement (I2) is immediate from $\I=\dcl{\B}$. For the proof of~(I3$'$), let $I$ and $I'$ be given as stated. Since $I\in\I=\dcl{\B}$, there exists a set $B_1\in\B$ such that $I\sub B_1$. Applying (B2$'$) to $B_1$ and $B_2:= I'\in\Imax\sub\B$ with $F_1 := B_1\sm (I\cup I')$, we find a set $F_2\sub I'\sm B_1$ such that $(B_1\sm F_1)\cup F_2\in\B$. For $B:= (B_1\sm F_1)\cup F_2$ we now have $I\sub B\sub I\cup I'$, as required for~(I3$'$).
 \endproof

For our alternative rank axioms we need to define a notion of $r$-independent sets. We do so as in~(RM). Thus, given any function ${r\colon (2^E\times 2^E)_\subseteq\to \N\cup\{\infty\}}$, a set $I\sub E$ is {\em $r$-independent\/} if $r(I|I-x)>0$ for every $x\in I$.

\begin{theorem}\label{Ralt}
A function $r\colon (2^E\times 2^E)_\subseteq\to \N\cup\{\infty\}$ satisfies the rank axioms if and only if $r$ satisfies the following four statements:\vskip-3pt\vskip-3pt
\begin{itemize}\itemsep=0pt
\item[\rm (R1)] For all $B\sub A\sub E$ we have $r(A|B)\le |A\sm B|$.
\item[\rm (R2)] For all $A,B\sub E$ we have $r(A|A\cap B)\ge r(A\cup B|B)$.
\item[\rm (R3)] For all $C\sub B\sub A\sub E$ we have $r(A|C) = r(A|B) + r(B|C)$.
\item[\rm (R4$'$)] For all $B\sub A\sub E$ there exist $r$-independent sets $J\sub I$ such that $J\sub B$ with $r(B|J)=0$ and $I\sub A$ with $r(A|I)=0$.%
   \COMMENT{}
\end{itemize}
\end{theorem}

\begin{proof} Let assume first that $r$ satisfies the rank axioms and prove~(R4$'$). Let $B\sub A\sub E$ be given. By~(RM), $B$~has a maximal $r$-independent subset~$J$, which extends by (RM) to a maximal $r$-independent subset $I$ of~$A$. Then $r(B|J)=0$ and $r(A|I)=0$ by (R4) and the maximalities of $J$ and~$I$.%
   \COMMENT{}

Assume now that $r$ satisfies (R1), (R2), (R3), (R4$'$). We will show that $r$ satisfies (R4) and (RM).

We first prove (R4). Suppose $A=\bigcup_{\gamma} A_\gamma$ and $B\sub A_\gamma$ for all $\gamma$ are such that $r(A|B)>0$. By (R4$'$), there exist $r$-independent sets $J\sub I$ such that $J\sub B$ with  $r(B| J)=0$ and $I\sub A$ with $r(A|I)=0$. Applying (R3) twice, we see that $r(I|J)=r(A|J)=r(A|B)>0$. Then $|I\sm J| > 0$ by~(R1); pick $x\in I\sm J$. Applying (R2) with $A'=J+x$ and $B'=I-x$,%
   \COMMENT{}
   we obtain $r(J+x\,|\,J)\geq r(I\,|\,I-x)>0$.  
As $x\in A$, we have $x\in A_{\gamma} $ for some $\gamma$. Now $r(A_{\gamma}|B)\geq r(A_{\gamma}|J)= r(A_{\gamma}|\,J+x)+r(J+x\,|\,J)>0$, as required.%
   \COMMENT{}

We now prove (RM). Let $J$ be an $r$-independent subset of some set $X\sub E$. Applying (R4$'$) with $A:=X$ and $B:=J$, we find $r$-independent sets $J'\sub I$ such that $J'\sub J$ with $r(J|J')=0$, and hence $J'=J$ by~(R3), and $I\sub X$ with $r(X,I)=0$. Again by~(R3), $I$~is a maximal $r$-independent subset of~$X$.
\end{proof}

The fact, shown above, that (R4$'$) implies both~(RM) (which is basically the special case of (R4$'$) that $B$ is $r$-independent) and~(R4) raises the question of whether (RM) might also imply~(R4)---in which case we could scrap (R4) in our rank axioms. A~simple example in Section~\ref{non-matroids} will show that it does not. The reason for why (R4$'$) is genuinely stronger than (RM) lies in the fact that, in the presence of~(R3), every $r$-independent subset $I$ of $X$ satisfying $r(X|I)=0$ must be maximally $r$-independent, but conversely a maximally $r$-independent subset $I$ of~$X$ does not automatically satisfy $r(X|I)=0$: the fact that is does, when $X\sm I$ is infinite, is precisely~(R4).

\medbreak

As mentioned earlier, we have no alternative systems of closure axioms. Oxley~\cite{Oxley92} proved that B-matroids are characterized by the closure axioms given in Section~\ref{axioms}, except that he replaced the usual (and our) axiom (CL4) by the stronger axiom
\begin{itemize}
\item[\rm (CL4$'$)] For all $Z\sub X\sub E$ and $y\in \cl(X)\sm \cl(Z)$ there exists an $x\in X\sm Z$ such that $x\in \cl((X-x)+y)$.
\end{itemize}
(For $|X\sm Z|=1$, axiom (CL4$'$) yields axiom~(CL4).)

The reason Oxley replaced (CL4) with (CL4$'$) was historical. Klee~\cite{Klee71}, in his own response to Rado's~\cite{Rado66matroids} challenge, had considered `C-matroids',%
   \COMMENT{}
   the systems of $\cl$-independent sets for functions $\cl\colon 2^E\to 2^E$ satisfying (CL1--3) and~(CL4$'$). Oxley~\cite{OxleyThesis}%
   \COMMENT{}
   had shown that these axioms are not strong enough to define a B-matroid,%
   \footnote{Here is another simple example. Let $\cl$ map every finite subset of $E$ to itself, and every infinite subset to~$E$. Then $\cl$ satisfies (CL1--3) and~(CL4$'$), but the \cl-independent sets are just the finite subsets of~$E$, which fail to satisfy~(M) for $X=E$.}
   and remedied this defect by adding~(CLM).

In the presence of~(CLM), however, the strengthening of (CL4) to (CL4$'$) becomes obsolete, because the first implies the second:

\begin{proposition}
If $\cl\colon 2^E\to 2^E$ satisfies the closure axioms, it even satisfies~{\rm(CL4$'$).}
\end{proposition}

\proof 
By Theorem~\ref{CLeq}\thinspace (ii), the set $\I$ of $\cl$-independent sets satisfies the independence axioms, and hence defines a B-matroid by Theorem~\ref{Bmatroids}. The closure operator associated with~$\I$, which by Theorem~\ref{CLeq}\thinspace (ii) is the function~\cl, satisfies (CL4$'$) by Oxley~\cite[Prop.\ 3.2.8]{Oxley92}. 
\endproof

Although we have no system of alternative circuit axioms without~(CM), there is a system of axioms for both circuits and cocircuits that achieves this, at least for countable matroids. These axioms, stated below, extend Minty's finite matroids axioms and are due to Bowler and Carmesin~\cite{BC13:Determinacy}.

Let us call these the {\em orthogonality axioms\/} for infinite matroids. Think of $\Ccal$ as the set of circuits, and $\Dcal$ as the set of cocircuits:

\begin{theorem}\label{Oaxioms}{\rm \cite{BC13:Determinacy}}
If $E$ is countable, two sets $\Ccal,\Dcal\sub 2^E$ are the sets of circuits and cocircuits of a matroid on $E$ if and only if they satisfy the following:\vskip-3pt\vskip-3pt
\begin{itemize}\itemsep=0pt
	\item[\rm (C1)] $\emptyset\notin \Ccal$
Ê Ê Ê Ê \item[\rm (C1$^*$)] $\emptyset\notin \Dcal$
	\item[\rm (C2)] No element of $\Ccal$ is a subset of another. Ê
Ê Ê Ê Ê \item[\rm (C2$^*$)] No element of $\Dcal$ is a subset of another.
Ê Ê Ê Ê \item[\rm (O1)] $|C\cap D|\neq 1$ for all $C\in \Ccal$ and $D\in \Dcal$.Ê
Ê Ê Ê Ê \item[\rm (O2)] For all partitions $E=P\cup Q\cup \{e\}$
either $P+e$ includes an element of $\Ccal$ containing~$e$ or
$Q+e$ includes an element of $\Dcal$ containing~$e$.
\item[\rm (O3)]For all $e\in C\in \Ccal$ and $X\subseteq E$, among all sets $C'\in \Ccal$ with $e\in C' \sub X \cup C$ there is one for which $C'\sm X$ is minimal.
\item [\rm (O3$^*$)]For all $e\in D\in \Dcal$ and $X\subseteq E$, among all sets $D'\in \Dcal$ with $e\in D' \sub X \cup D$ there is one for which $D'\sm X$ is minimal.
\end{itemize}
\end{theorem}

\noindent
It is not known whether the countability assumption in this theorem is necessary.

\medskip
In the presence of (O1) and~(O2), the statements (O3) and~(O3$^*$) follow (even for uncountable~$E$) from {\it tameness\/}, the property that every intersection of a circuit and a cocircuit is finite. Countable tame matroids, therefore, have a particularly simple axiomatization, the only non-trivial one of which to check is usually~(O2):

\begin{theorem}[Bowler \& Carmesin~\cite{BC13:Determinacy}]\label{Taxioms}
If $E$ is countable, two sets $\Ccal,\Dcal\sub 2^E$ are the sets of circuits and cocircuits of a tame matroid on $E$ if and only if they satisfy the following:\vskip-3pt\vskip-3pt
\begin{itemize}\itemsep=0pt
	\item[\rm (C1)] $\emptyset\notin \Ccal$
Ê Ê Ê Ê \item[\rm (C1$^*$)] $\emptyset\notin \Dcal$
	\item[\rm (C2)] No element of $\Ccal$ is a subset of another. Ê
Ê Ê Ê Ê \item[\rm (C2$^*$)] No element of $\Dcal$ is a subset of another.
Ê Ê Ê Ê \item[\rm (O1)] $|C\cap D|\neq 1$ for all $C\in \Ccal$ and $D\in \Dcal$.Ê
Ê Ê Ê Ê \item[\rm (O2)] For all partitions $E=P\cup Q\cup \{e\}$
either $P+e$ includes an element of $\Ccal$ containing~$e$ or
$Q+e$ includes an element of $\Dcal$ containing~$e$.
\item[\rm (T)] $|C\cap D|$ is finite for all $C\in \Ccal$ and $D\in \Dcal$.
\end{itemize}
\end{theorem}

The most useful way to apply Theorem~\ref{Taxioms} is that (O1), (O2) and (T) are enough to generate a matroid: if we have sets $\Ccal$ and $\Dcal$
satisfying these axioms then their minimal nonempty members give the
circuits and cocircuits of a matroid~\cite{BC13:Determinacy}.

Given a set $\Ccal\subseteq 2^E$, let us write $\Ccal^\perp$ for the set of subsets of $E$ that meet no element of $\Ccal$ exactly once. (Thus, (O1) says that $\Dcal$ is a subset of~$\Ccal^\perp$, or equivalently that $\Ccal$ is a subset of~$\Dcal^\perp$.) Bowler and Carmesin~\cite{BC13:Determinacy} have shown that $\Ccal$ and $\Ccal^\perp$ satisfy (O2) if and only if
$\Ccal$ satisfies the circuit elimination axiom~$(C3)$.

\section{Examples of non-matroids}\label{non-matroids}

In this section we illustrate our axioms by examples of set systems that narrowly fail to be matroids, by missing just one axiom of a given set. In particular, the axioms within each system will be seen to be independent.

We start with an example mentioned before:

\begin{example}\label{finitesets}
Let $E$ be infinite. The set $\I$ of finite subsets of $E$ satisfies {\rm(I1)--(I3)}, even~{\rm(I3$'$).} But $\I$  has no maximal element, so it violates {\rm(IM)} and~{\rm(I1$'$)}.%
   \COMMENT{}
\end{example}

Our next example shows that, although we can now deal with infinite sets, matroids still live in the discrete world:

\begin{example}\label{topclosure}
The usual topological closure operator for subsets of $\R$ satisfies the closure axioms {\rm(CL1)--(CL4)} for $E:=\R$, but not~{\rm(CLM).}
\end{example}

\proof
To see that (CLM) fails, notice that independent sets must be discrete. Hence there is no maximal such set in~$\R$.%
   \COMMENT{}
\endproof

We continue with two examples showing that neither of our rank axioms (R4) and (RM) implies the other, given (R1)--(R3).

\begin{example}\label{RMfails}
Let $E$ be infinite. Given $B\sub A\sub E$ define $r(A|B):= |A|-|B|$, with $\infty-\infty:=0$. Then $r$ satisfies {\rm(R1)--(R4)} but not~{\rm (RM)}.
\end{example}

\proof
(RM) fails, because the $r$-independent sets are precisely the finite sets.
\endproof

\begin{example}\label{R4fails}
Let $E$ be infinite. Given $B\sub A\sub E$, let $r(A|B)$ be $0$ if $A$ and $B$ are either both finite or both infinite, and $1$ otherwise. Then $r$ satisfies {\rm(R1)--(R3)} and {\rm(RM),} but not~{\rm (R4)}.
\end{example}

\proof
(RM) holds, since $\es$ is the only $r$-independent set. To see that (R4) fails, let $B:=\es$ and consider as $(A_\gamma)_\gamma$ an infinite increasing chain of finite sets.
\endproof

Our motivation for the alternative axiom systems given in Section~\ref{alt} was to replace our axiom~(M) with something weaker. It led us to replace it by~(I1$'$), while strengthening the extension axiom (I3) to~(I3$'$) or the base exchange axiom (B2) to~(B2$'$). We now show that (I1$'$) cannot be weakened further on the basis of (I2) and~(I3$'$), the other alternative independence axioms from Theorem~\ref{Ialtdel}.

Example~\ref{finitesets} shows that we cannot replace (I1$'$) with~(I1): it is not enough that some independent set exists, we need a maximal one. But then (I3$'$) gives us many more. This led Higgs~\cite{Higgs69equicard} to ask whether our set of axioms from Theorem~\ref{Ialtdel}, with
\begin{itemize}
\item[\rm (I0)] $\I$ has a maximal element
\end{itemize}

\noindent
instead of~(I1$'$), would yield a (B-) matroid. If that was the case, then the axiom (M) common to all our systems would also be too strong: we could extract its `extension' part as~(I3$'$) and limit its `existence' part to~(I0). Of course, we would still want restrictions of matroids to be matroids, so (I1$'$)---which is (I0) for restrictions---should, somehow, follow.

Our next example shows that it does not: the example satisfies all axioms other than those of type (M) or type~1$'$,%
   \COMMENT{}
   including~(I0), but the set $X\sub E$ (see below) has no maximal independent subset.

To define this example, let $X = \{x_0,x_1,\dots\/\}$ and $Y = \{y_0,y_1,\dots\}$ be disjoint infinite sets, and let $G$ be the graph with vertex set $E:= X\cup Y$ and edge set $\{\,x_iy_i\mid i\in\N\,\}$, an infinite matching. Let $\B$ be the set of all transversals of the edges that meet $X$ only finitely, i.e., of all subsets $U$ of $X\cup Y$ satisfying $|U\cap\{x_i,y_i\}| = 1$ for every~$i$ and $|U\cap X| < \infty$ (Fig.~\ref{matching}). Let us call the elements of $\B$ the {\em skew transversals\/} of~$G$. Put $\I:=\dcl{\B}$, call its elements {\em independent\/} and the elements of $2^E\sm\I$ {\em dependent\/} sets. Let $\C$ be the set of the minimal dependent sets, or {\em circuits\/}, and let cl be the closure operator associated with~$\I$.

\begin{figure}[htbp]
  \centering
  \includegraphics[width=0.7\linewidth]{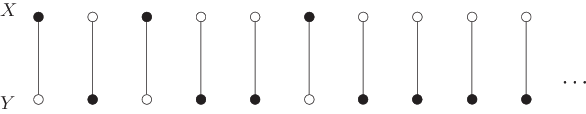}
  \caption{The set of black vertices is an element of~$\B$.}
  \label{matching}
\end{figure}

\begin{example}\label{example2}
Let $\B$ be the set of skew transversals of the graph $G$ shown in Figure~\ref{matching}, and let $\I$, $\C$ and $\cl$ be defined as above.\vskip-3pt\vskip-3pt
\begin{enumerate}[\rm(i)]\itemsep=0pt
\item $\I$ satisfies {\rm(I0)--(I3)} and {\rm(I3$'$),} but not~{\rm(I1$'$)} or {\rm(IM)}.
\item $\B$ satisfies {\rm(B0)--(B2)} and~{\rm(B2$'$)}, but not~{\rm(B1$'$)} or {\rm(BM)}.
\item Not every dependent set contains a circuit.
\item $\C$ satisfies {\rm(C1)--(C3)}, but not~{\rm(CM)}.
\item The operator $\cl$ satisfies {\rm(CL1)--(CL4)}, but not~{\rm(CLM)}.\label{clitem}
\item There is no relative rank function associated with~$\I$.
\end{enumerate}
\end{example}

\proof
There are two properties of a set $U\sub E$ that will each make it dependent: meeting $X$ infinitely, and containing an edge $\{x_i,y_i\}$. Every single edge is a circuit, while dependent transversals meeting $X$ infinitely contain no circuit. So the circuits are precisely the single edges. $X$~itself, then, is dependent but contains no circuit, showing~(iii). Its independent subsets are its finite subsets. It thus has no maximal independent subset and hence violates (I1$'$) and~(B1$'$). This proves (i) and~(ii). (The other axioms are easy to check.) Since no two distinct circuits meet, (C3)~holds vacuously. Since $X$ violates~(M) for~$\I$, none of (IM), (BM), (CM) or (CLM) holds. Finally, the closure $\cl(U)$ of a set $U$ is obtained by adding to $U$ all those $x_i$ for which $y_i\in U$, and all those $y_i$ for which $x_i\in U$.%
   \COMMENT{}
   Statement~\eqref{clitem} follows. Since $X$ does not contain a maximal independent set, $r(A|X)$ would not be well-defined for any $A\supseteq X$, and so we have (vi).
\endproof

As we saw in Section~\ref{examples}, the elementary algebraic cycles of a graph~$G$, the edge sets of its finite cycles and double rays, are the circuits of a matroid if and only if $G$ contains no subdivision of the Bean graph (Fig.~\ref{Beangraph}). This means that, at least in this class of examples,%
   \COMMENT{}
   the Bean graph `only just' fails to define a matroid. And indeed, we shall see below that its elementary algebraic cycles again violate exactly one of the axioms in each set.

To prove this, we need a formal definition of the Bean graph. Let it be the graph consisting of two disjoint rays $R,S$ with starting vertices $u\in R$ and $v\in S$ and all possible edges from $v$ to~$R$. Write $y$ for the edge~$uv$, and pick an edge $z\in E(S)$. 
Let $\C$ be the set of the elementary algebraic cycles of this graph, and call them {\em circuits\/}. Write $\I$ for the set of edge sets not containing a circuit, put $\B:=\Imax$, let $\cl$ be the closure operator and $r$ any rank function associated with~$\I$, and call the sets in $2^{E(G)}\sm\I$ {\em dependent\/}.

\begin{example}\label{Bean}
Let $\C$ be the set of elementary algebraic cycles of the Bean graph, and let $\I$, $\B$, $\cl$ and $r$ be as defined above.\vskip-3pt\vskip-3pt
\begin{enumerate}[\rm(i)]\itemsep=0pt
\item $\C$ satisfies {\rm(C1)}, {\rm(C2)}, {\rm(CM)} and the usual finite circuit elimination axiom, but not the infinite elimination axiom~{\rm(C3)}.
\item Every dependent set contains a circuit.
\item $\I$ satisfies {\rm(I0), (I1) (}even {\rm(I1$'$)),} {\rm(I2)} and {\rm(IM)}, but not~{\rm(I3)} or {\rm(I3$'$)}.
\item $\B$ satisfies {\rm(B0)}, {\rm(B1) (}even {\rm(B1$'$))} and {\rm(BM)}, but not~{\rm(B2)} or {\rm(B2$'$)}.
\item The operator $\cl$ satisfies {\rm(CL1)}, {\rm(CL2)}, {\rm(CL4)} and {\rm(CLM)}, but not~{\rm(CL3)}.%
\item The function $r$ satisfies {\rm (R1)}, {\rm (R2)}, {\rm (R4)}, {\rm (R4$'$)} and {\rm (RM)}, but not {\rm (R3)}.
\end{enumerate}
\end{example}

\proof
(i) Assertions (C1) and~(C2) are trivial. The circuit elimination axiom for just two circuits is easily checked by straightforward case analysis for the current graph, or proved in general as for finite graphs by considering vertex degrees.%
   \COMMENT{}
   For a proof of~(CM), let $X$ be any set of edges, and let $I\subseteq X$ be an independent set. Using Zorn's Lemma we find a maximal subset $B$ of $X$ that contains $I$ but contains no finite circuit. (Such a set $B$ consists of  a spanning tree in each component of the corresponding subgraph). If $B$ contains no double ray either, it is a maximal independent subset of~$X$ containing $I$. 
If it does, then this double ray $D$ is unique: since any double ray in the Bean graph has to link its two ends, two distinct double rays would form a finite cycle between them. As $I$ is independent, $D$~contains an edge $x\notin I$. By the circuit elimination axiom for two circuits, $B-x$ is a maximal independent subset of~$X$ as required by~{\rm(CM)}.%
   \COMMENT{}

To see that (C3) fails, consider the circuit $C := E(R)\cup E(S) +y$, its subset $X:= E(R)$, the edge $z\in E(S)\sub C$, and for every $x\in X$ as $C_x$ the unique triangle containing~$x$.

(ii) This is trivial.

(iii) While (I0) is easy (consider $E(R)\cup E(S) + y - z$), (I1)~and (I2) are trivial. We have already proved~(IM), which implies~(I1$'$). To see that (I3) and (I3$'$) fail, consider as $I$ the set of unbroken edges in Figure~\ref{Beangraph}, and as~$I'$ the set $E(R)\cup E(S) + y - z$. Then $I\in \I\sm\Imax$ (since we can add $z$ and remain independent), while $I'\in\Imax$. But, clearly, $I$~does not extend properly to any independent subset of $I\cup I'$.%
   \COMMENT{}

(iv) $\B$ is the set of spanning trees not containing a double ray. It clearly satisfies (B0) and~(B1). We have already proved~(BM), and this implies~(B1$'$). To see that (B2) fails (and with it~(B2$'$)), consider as $B_1$ the set of unbroken edges in Figure~\ref{Beangraph} plus the edge~$z$, and as $B_2$ the set $E(R)\cup E(S) + y - z$. Then we cannot delete $z$ from $B_1$, add an edge of $B_2\sm B_1$, and remain independent.

(v) While (CL1), (CL2) and~(CL4)%
   \COMMENT{}
   are trivial, we have already proved~(CLM). To see that (CL3) fails, consider as $X$ the set of unbroken edges in Figure~\ref{Beangraph}. Its closure $\cl(X)$ contains all the broken edges except~$z$,%
   \COMMENT{}
   but $\cl(\cl(X))$ contains $z$ as well.%
   \COMMENT{}
   
(vi) (R1) is trivial, and (R2) derives from (IM) and~(I2).%
   \COMMENT{}
   With $B_2$ as above, we have $r(E|B_2)=0\neq 1+0=r(E\,|\,E-z)+r(E-z\,|\,B_2)$, so (R3) fails. To see that (R4) holds, note that $r(A|B)=0$ if and only if $J+x$ contains a circuit for each maximal $J$ in $\I\cap 2^B$ and each $x\in A\sm B$. (R4$'$) is straightforward from~(RM), proved below, and the definition of $r$. Finally, if $I\subseteq E$, then $I$ contains a circuit if and only if $r(I\,|\,I-x)=0$ for some $x\in I$. Thus the set of $r$-independent sets equals $\I$, and hence (RM) holds.
\endproof

\bibliographystyle{amsplain}
\bibliography{collective}
\small
\vskip2mm plus 1fill
Version 23 February 2013

\end{document}